\documentclass{amsart}
\usepackage{amssymb,mathrsfs,MnSymbol}
\usepackage{color}

\def\bbb {\mathbf{b}}
\def\bC {\mathbf{C}}

\def\bN {\mathbf{N}}
\def\bp {\mathbf{p}}

\def\bq {\mathbf{q}}

\def\bR {\mathbf{R}}

\def\bV {\mathbf{V}}
\def\bW {\mathbf{W}}

\def\fH {\mathfrak{H}}
\def\fS {\mathfrak{S}}

\def\cB {\mathcal{B}}
\def\cC {\mathcal{C}}
\def\cD {\mathcal{D}}
\def\cE {\mathcal{E}}
\def\cF {\mathcal{F}}

\def\cL {\mathcal{L}}
\def\cM {\mathcal{M}}

\def\cP {\mathcal{P}}

\def\cS {\mathcal{S}}

\def\cV {\mathcal{V}}
\def\cW {\mathcal{W}}

\def\scrB{\mathscr{B}}
\def\scrH{\mathscr{H}}

\def\scrR{\mathscr{R}}
\def\scrV{\mathscr{V}}

\def\a {{\alpha}}
\def\b {{\beta}}
\def\g {{\gamma}}
\def\Ga {{\Gamma}}
\def\de {{\delta}}

\def\l {{\lambda}}
\def\L {{\Lambda}}
\def\si {{\sigma}}

\def\om {{\omega}}

\def\d {{\partial}}
\def\grad {{\nabla}}
\def\Dlt {{\Delta}}

\def\la {\langle}
\def\ra {\rangle}

\def\bu {{\bullet}}
\def\dd {\,\mathrm{d}}

\newcommand{\Div}{\operatorname{div}}

\newcommand{\Tr}{\operatorname{trace}}

\newcommand{\ad}{\operatorname{\mathbf{ad}}}

\def\OP {\mathrm{OP}}

\newcommand{\ba}{\begin{aligned}}
\newcommand{\ea}{\end{aligned}}

\newcommand{\be}{\begin{equation}}
\newcommand{\ee}{\end{equation}}

\newcommand{\lb}{\label}

\newtheorem{Thm}{Theorem}[section]

\newtheorem{Prop}[Thm]{Proposition}

\newtheorem{Lem}[Thm]{Lemma}
\newtheorem{Def}[Thm]{Definition}

\begin{document}

\title[Empirical Measures and Quantum Mechanics]{Empirical Measures and Quantum Mechanics:\\ Application to the Mean-Field Limit}

\author[F. Golse]{Fran\c cois Golse}
\address[F.G.]{CMLS, \'Ecole polytechnique, CNRS, Universit\'e Paris-Saclay , 91128 Palaiseau Cedex, France}
\email{francois.golse@polytechnique.edu}

\author[T. Paul]{Thierry Paul}
\address[T.P..]{CMLS, \'Ecole polytechnique, CNRS, Universit\'e Paris-Saclay , 91128 Palaiseau Cedex, France}
\email{thierry.paul@polytechnique.edu}

\begin{abstract}
In this paper, we define a quantum analogue of the notion of empirical measure in the classical mechanics of $N$-particle systems. We establish an equation governing the evolution of our quantum analogue of the $N$-particle empirical 
measure, and we prove that this equation contains the Hartree equation as a special case. Our main application of this new object to the mean-field limit of the $N$-particle Schr\"odinger equation is an $O(1/\sqrt{N})$ convergence rate 
in some appropriate dual Sobolev norm for the Wigner transform of the single-particle marginal of the $N$-particle density operator, uniform in $\hbar\in(0,1]$ (where $\hbar$ is the Planck constant) provided that $V$ and $(-\Dlt)^{3+d/2}V$ 
have integrable Fourier transforms.
\end{abstract}



\date{\today}

\maketitle



\section{Introduction and Main Result}


\subsection{The Mean-Field Limit for the Dynamics of $N$ Identical Particles}


In classical mechanics, the dynamics of $N$ identical, interacting point particles is governed by Newton's second law of motion written for each particle. One obtains in this way a system of $6N$ coupled ordinary differential equations, 
set on a $6N$-dimensional phase space. For large values of $N$, solving such a differential system becomes impracticable. One way of reducing the complexity of this problem is to solve the Liouville equation for the phase space
number density of the ``typical particle'', replacing the force exerted on that particle by the ``self-consistent'' force, also called the ``mean-field force'', computed in terms of the solution of the Liouville equation itself. The mean-field
equation so obtained is the Vlasov equation (see equation\eqref{Vlasov} below). When the force field derives from a $C^{1,1}$ potential, the validity of this approximation has been proved in \cite{NeunzertWick, BraunHepp, Dobrushin}. 
The case of the Coulomb (electric), or of the Newton (gravitational) potential remains open at the time of this writing, in spite of significant progress on this program: see \cite{HaurayJabin1, HaurayJabin2} for singularities weaker than
the Coulomb or Newton $1/r$ singularity at the origin, or \cite{PicklLazaro,Lazaro} for the Coulomb potential with a vanishing regularization as the particle number $N$ tends to infinity. All these results are based on a remarkable 
property, stated as Proposition \ref{P-WSolVlasov} below, of the phase-space empirical measure of the $N$-particle system governed by a system if coupled ordinary differential equations. See also \cite{MischMouWenn} for another 
approach, which avoids using the nice properties of this empirical measure, and applies to more general dynamics (involving jump or diffusion processes for instance).

In quantum mechanics, the analogous mean-field approximation goes back to the work of Hartree \cite{Hartree1928}. Rigorous derivations of the time-dependent Hartree equation from the quantum $N$-body problem have been obtained
in the case of bounded potentials by Hepp \cite{Hepp74} (using coherent states), then by Spohn \cite{Spohn80} (for pure states, using the BBGKY hierarchy) --- see also \cite{BardosFGMauser} for a discussion of the case of mixed states 
following Spohn's very concise argument. The case of singular potentials, including the Coulomb potential, is treated in \cite{ErdosYau} (see also \cite{BEGMY}) in terms of the BBGKY hierarchy, in \cite{Pickl,KnowlesPickl} by a simpler
argument in the case of pure states, and in \cite{RodSchlein,ChenLeeSchlein,BenePortaSchlein} with second quantization techniques. All these results assume that the value of the Planck constant is kept fixed while $N$ tends to infinity. 
This is also true in the special case of the Fermi-Dirac statistics which involves a mean-field scaling of the interaction leading,  after some appropriate rescaling of time, to an \textit{effective} Planck constant $\hbar\sim N^{-1/3}$: see 
\cite{NarnhoSewell, BeneJakPortaSaffiSchlein}.

On the other hand, the system of Newton's equations governs the asymptotic behavior of the quantum $N$-body problem for all $N$ kept fixed as $\hbar$ tends to zero. Equivalently, the asymptotic behavior of solutions of the Heisenberg 
equation in the classical limit is described by the Liouville equation. Since the Heisenberg equation governs the evolution of time-dependent density operators on some Hilbert space, while the Liouville equation describes the dynamics
of a probability density in phase space, the classical limit can be formulated in terms of the Wigner function, associated to any quantum observable as recalled in Appendix B. There is a huge literature on this subject; see for instance
\cite{LionsPaul} for a proof of the weak convergence of the Wigner function of a solution of the Heisenberg equation to a positive measure, solution of the Liouville equation, in the limit as $\hbar\to 0$. This result holds true for all $C^{1,1}$
potentials and a very general class of initial data (without any explicit dependence in $\hbar$). Likewise, the Vlasov equation governs the asymptotic behavior of solutions of the Hartree equation (equation \eqref{Hartree} below) for a very
large class of potentials including the Coulomb case: see again \cite{LionsPaul} for a result formulated in terms of weak convergence of Wigner functions. For regular potentials, more precise information on the convergence of Wigner
functions can be found in \cite{APPP1,APPP2,BPSS}.

In other words, the mean-field limit has been established rigorously both for each fixed $\hbar>0$ and for the vanishing $\hbar$ limit of the quantum $N$-particle dynamics independently. This suggests the problem of obtaining a uniform 
in $\hbar\in(0,1]$ convergence rate estimate for the mean-field limit of the quantum $N$-particle dynamics. A positive answer to this question, valid without restriction on the initial data, would justify in particular using the Vlasov instead 
of the Hartree equation for large systems of heavy particles.

The first results in this direction are \cite{PezzoPulvi} (where the mean-field limit is established term by term in the semiclassical expansion) and \cite{GraffiMartiPulvi} for WKB states along distinguished limits of the form $N\to+\infty$ 
with $\hbar\equiv\hbar(N)\to 0$. See also \cite{FrohGraffiSchwartz} for results in the case of very special interactions, not defined in terms of a potential.

If both the interaction potential and the initial data are analytic, the BBGKY approach leads to a uniform in $\hbar\in(0,1]$ convergence rate estimate for the mean-field limit, of optimal order $O(1/N)$: see \cite{FGPaulPulvi}. This $O(1/N)$ 
bound, obtained at the cost of stringent, physically unsatisfying regularity assumptions, is similar to the nonuniform in $\hbar$ convergence rate obtained in \cite{ChenLeeSchlein} --- except the latter result holds for singular potentials 
including the Coulomb case.

A new approach to the uniformity problem has been proposed in \cite{FGMouPaul, FGPaul}; it is based on a quantum analogue of the quadratic Monge-Kantorovich-Wasserstein distance, similar to the one used in \cite{Dobrushin}.
This method provides an estimate of the convergence rate in the mean-field limit that is uniform as $\hbar\to 0$. This method can be combined with the usual BBGKY strategy, following carefully the dependence of the error estimate
in terms of $\hbar$, to produce a uniform in $\hbar\in(0,1]$ convergence rate of order $O((\ln\ln N)^{-1/2})$ for  initial data of semiclassical type: see \cite{FGPaulPulvi}.

\smallskip
The main  results of the present article are Theorem \ref{T-MFUnif}, and Theorem \ref{T-EMQEq} together with Definition \ref{D-EMQ}. Theorem \ref{T-MFUnif}  establishes the quantum mean-field limit with rate of convergence of order 
$\frac1{\sqrt N}$, uniformly in  $\hbar\in(0,1]$ and for all initial factorized data. By ``all", we mean that any dependence in the Planck constant is allowed. In particular, no semiclassical scaling of the initial data is required. The proof of 
Theorem \ref{T-MFUnif} requires defining a notion of empirical ``measure'' in quantum mechanics. The precise definition of this new (to the best of our knowledge) mathematical object can be found in Definition \ref{D-EMQ} of Section 
\ref{S-EqEM}, and the equation governing its dynamics is derived in Theorem \ref{T-EMQEq}.

\smallskip
Before stating this convergence rate estimate, we need to introduce some notation used systematically in the present paper. 

Denote by $\fH:=L^2(\bR^d)$ (the single-particle Hilbert space) and, for each $N\ge 1$, let $\fH_N:=\fH^{\otimes N}\simeq L^2((\bR^d)^N)$. Henceforth, $\cL(\fH)$ (resp. $\cL^1(\fH)$) designates the space of bounded (resp. trace-class) 
operators on $\fH$. For each permutation $\si\in\fS_N$ (the group of permutations of $\{1,\ldots,N\}$), let $U_\si$ be the unitary operator on $\fH_N$ defined by the formula
$$
(U_\si\Psi_N)(x_1,\ldots,x_N):=\Psi_N(x_{\si^{-1}(1)},\ldots,x_{\si^{-1}(N)})\,.
$$
We denote by $\cL_s(\fH_N)$ (resp. $\cL^1_s(\fH_N)$) the set of bounded (resp. trace-class) operators $F_N$ on $\fH_N$ satisfying the condition
$$
U_\si F_NU^*_\si=F_N\,,\qquad\hbox{ for all }\si\in\fS_N\,.
$$
We denote by $\cD(\fH)$ the set of density operators on $\fH$, i.e. operators $R$ satisfying
$$
R=R^*\ge 0\,,\qquad\Tr_\fH(R)=1\,.
$$
Likewise, we denote by $\cD_s(\fH_N)$ the set of symmetric $N$-particle densities on $\fH_N$, i.e. $\cD_s(\fH_N):=\cD(\fH_N)\cap\cL_s(\fH_N)$. 

\smallskip
Let $R_N\in\cD_s(\fH_N)$; set $r_N\equiv r_N(x_1,\ldots,x_N;y_1,\ldots,y_N)$ to be the integral kernel of $R_N$. The first marginal of the quantum density $R_N$ is the element $R_{N:1}$ of $\cD(\fH)$ whose integral kernel is
$$
r_{n:1}(x,y):=\int_{(\bR^d)^{N-1}}r_{N:1}(x,z_2,\ldots,z_N,y,z_2,\ldots,z_N)\dd z_2\ldots\dd z_N\,.
$$

\smallskip
For each integer $n\ge 0$, we denote by $C^n_b(\bR^d\times\bR^d)$ the set of functions $f$ of class $C^n$ on $\bR^d\times\bR^d$ such that $\d^\a f$ is continuous and bounded on $\bR^d\times\bR^d$ for all $\a\in\bN^d$ with
$|\a|\le n$. This is a Banach space for the norm
$$
\|f\|_{n,\infty}:=\max_{|\a|\le n}\|\d^\a f\|_{L^\infty(\bR^d\times\bR^d)}\,.
$$
Finally, let $\|\cdot\|'_{n,\infty}$ be the norm of the topological dual of $C^n_b(\bR^d\times\bR^d)$. In other words, for each continuous linear functional $L$ on $C^n_b(\bR^d\times\bR^d)$, one has
$$
\|L\|'_{n,\infty}:=\sup\{|\la L,f\ra|\,\,:\,\,\|f\|_{n,\infty}\le 1\}\,.
$$

\begin{Thm}\lb{T-MFUnif}
Assume that $V\in C_0(\bR^d)$ is an even, real-valued function whose Fourier transform $\hat V$ satisfies
$$
\bV:=\tfrac1{(2\pi)^d}\int_{\bR^d}|\hat V(\xi)|(1+|\xi|)^{d+6}\dd\xi<\infty\,.
$$
Let $R^{in}\in\cD(\fH)$, and let $t\mapsto R(t)$ be the solution to the Cauchy problem for the Hartree equation
\be\lb{Hartree}
i\hbar\d_tR(t)=[-\tfrac12\hbar^2\Dlt+V_{R(t)},R(t)]\,,\qquad R(0)=R^{in}
\ee
with mean-field potential 
\be\lb{MFPot}
V_{R(t)}(x):=\int_{\bR^d}V(x-z)r(t,z,z)\dd z\,,
\ee
where $r\equiv r(t,x,y)$ is the integral kernel of $R(t)$.

For each $N\ge 2$, let $t\mapsto F_N(t)$ be the solution to the Cauchy problem for the Heisenberg-von Neumann equation
\be\lb{NHeisen}
i\hbar\d_tF_N(t)=[\scrH_N,F_N(t)]\,,\qquad F_N(0)=F^{in}_N
\ee
with initial data $F^{in}_N=(R^{in})^{\otimes N}$, where the $N$-body quantum Hamiltonian is
\be\lb{NHamilt}
\scrH_N:=\sum_{j=1}^N-\tfrac12\hbar^2\Dlt_{x_k}+\frac1N\sum_{1\le j<k\le N}V(x_j-x_k)\,.
\ee

Then, for all $\hbar\in(0,1]$, all $N\ge 1$ and all $t\ge 0$, the Wigner transforms at scale $\hbar$ of $F_{N:1}(t)$ and $R(t)$ satisfy the bound
\be\lb{UnifCvRate}
\left\|W_{\hbar}[F_{N:1}(t)]-W_{\hbar}[R(t)]\right\|'_{d+5,\infty}\le\frac{\g_d+1}{\sqrt{N}}\exp\left(C_dte^{(d+5)\Ga t}\bV\left(1\!+\!\hbar^2\tfrac{\bV}{(d+5)\Ga}\right)\right)\,,
\ee
where $\Ga:=\|\grad^2V\|_{L^\infty}$ while $\g_d$ and $C_d\equiv C_d[\bV]$ are positive constants which depend only on the space dimension $d$, and on $d$ and $\bV$ respectively.
\end{Thm}

\smallskip
The definition and the elementary properties of the Wigner transform are recalled in the Appendix. We recall that, in the classical limit, i.e. for $\hbar\to 0$, quantum densities propagated by the Heisenberg-von Neumann equations will 
typically fail to converge to any limiting density operator. However, up to extracting subsequences $\hbar_n\to 0$, the corresponding sequence of Wigner transforms will have limit points in $\cS'(\bR^d\times\bR^d)$, and the classical 
limit of quantum mechanics can be described in terms of the evolution of these limit points. See \cite{LionsPaul} for a presentation of this description of the classical limit (especially Theorems III.1 and IV.1 in \cite{LionsPaul}).

\smallskip
The $O(1/\sqrt{N})$ estimate in this result is on a par with the convergence rate obtained in \cite{RodSchlein} --- except \eqref{UnifCvRate} is uniform in $\hbar\in(0,1]$, unlike the bound in \cite{RodSchlein}. On the other hand, one should
keep in mind that the regularity assumptions on the potential are more stringent in Theorem \ref{T-MFUnif} than in \cite{RodSchlein}, which can handle singular potentials, including the Coulomb case. 

Notice also the double exponential growth in time of the bound in Theorem \ref{T-MFUnif}, to be compared either with the simple exponential growth in the nonuniform in $\hbar$ convergence rate in \cite{RodSchlein}, or in the uniform
as $\hbar\to 0$ estimate in \cite{FGMouPaul}. 

A last difference between the convergence rate obtained in \cite{RodSchlein, ChenLeeSchlein}  and the bound in Theorem \ref{T-MFUnif} lies in the metric used to measure the difference between the Hartree solution and the first marginal
of the $N$-body density operator: the estimate in \cite{RodSchlein,ChenLeeSchlein} is expressed in terms of the trace norm, whereas \eqref{UnifCvRate} is formulated in terms of Wigner transforms, whose difference is measured in 
some dual Sobolev norm. In the classical limit, quantum particles are expected to be perfectly localized on phase-space curves, and Schatten norms (including the trace norm) are ill-suited to capturing this asymptotic behavior --- see 
section 5 of \cite{FGPaulCohSta} for a detailed discussion of this point.

\subsection{Empirical Measure of $N$-Particle Systems in Classical Mechanics}


As mentioned above, an important ingredient in the rigorous derivation of the Vlasov equation from the $N$-particle system of Newton's equations in classical mechanics (see \cite{BraunHepp,Dobrushin}) is the following remarkable 
property of the empirical measure, which is briefly recalled below for the reader's convenience.

Consider the system of Newton's equations of motion for a system of $N$ identical point particles of mass $1$, with pairwise interaction given by a à (real-valued) potential $V/N$ assumed to be even and smooth (at least $C^2$) on $\bR^d$, 
and such that $V$, $\grad V$ and $\grad^2V$ are bounded on $\bR^d$:
\be\lb{NBodyC}
\left\{
\ba
\dot{x}_k&=\xi_k\,,&&\qquad x_k(0)=x_k^{in}\,,
\\	\\
\dot\xi_k&=-\frac1N\sum_{l=1}^N\grad V(x_k-x_l)\,,&&\qquad\xi_k(0)=\xi_k^{in}\,.
\ea
\right.
\ee
(The term corresponding to $l=k$ in the sum is equal to $0$ since $V$ is even). We denote by $T^N_t$ the flow defined by the differential system (\ref{NBodyC}).

Set $z_k:=(x_k,\xi_k)$ and $Z_N:=(z_1,\ldots,z_N)$. To each $N$-tuple $Z_N\in(\bR^d\times\bR^d)^N$, one associates the empirical measure
\be\lb{DefEmpMeas0}
\mu_{Z_N}:=\frac1N\sum_{k=1}^N\de_{z_k}\,.
\ee
One defines in this way a map
$$
(\bR^d\times\bR^d)^N\ni Z_N\mapsto\mu_{Z_N}\in\cP(\bR^d\times\bR^d)
$$
(where $\cP(X)$ designates the set of Borel probability measures on $X$). The system (\ref{NBodyC})  takes the form
$$
\dot{x}_k=\xi_k\,,\qquad\dot\xi_k=-\int_{\bR^d\times\bR^d}\grad V(x_k-y)\mu_{Z_N}(\dd y\dd\eta)\,,\qquad 1\le k\le N\,,
$$
and this implies the following remarkable result.

\begin{Prop}\lb{P-WSolVlasov}
For each $N\ge 1$, the map $t\mapsto\mu_{T^N_tZ_N}$ is continuous on $\bR$ with values in $\cP(\bR^d\times\bR^d)$ equipped with the weak topology, and is a weak solution (in the sense of distributions) of the Vlasov equation
\be\lb{Vlasov}
\d_tf+\xi\cdot\grad_xf-\Div_\xi\left(f(\grad_xV\star_{x,\xi}f)\right)=0\,.
\ee
\end{Prop}

\smallskip
With Proposition \ref{P-WSolVlasov}, the derivation of the Vlasov equation (\ref{Vlasov}) from Newton's equations (\ref{NBodyC}) for $N$-particle systems is equivalent to the continuous dependence of weak solutions of (\ref{Vlasov}) 
on the initial data for the weak topology of Borel probability measures on the single particle phase-space. 

In view of the conceptual simplicity of the approach of the mean-field limit in classical mechanics based on the empirical measure (see \cite{NeunzertWick,BraunHepp,Dobrushin}), it seems that a similar structure on the quantum 
$N$-body problem would be extremely helpful for the purpose of deriving the mean-field limit with a uniform convergence rate in the Planck constant $\hbar$.

A notion of empirical measure in quantum mechanics will be introduced in Definition \ref{D-EMQ} below. The equation governing its evolution --- equation \eqref{EMQEq} --- will be derived in Theorem \ref{T-EMQEq}. The Hartree 
equation can be viewed as a ``special case'' of equation \eqref{EMQEq}. More precisely, Theorem \ref{T-Hartree} states that solutions to the Hartree equation can be identified with a special class of solutions to equation \eqref{EMQEq}.

\smallskip
There is an obvious difficulty with the problem addressed in Theorems \ref{T-EMQEq} and \ref{T-Hartree}: indeed, the proof of Proposition \ref{P-WSolVlasov} is based on the method of characteristics for solving transport equations 
such as \eqref{Vlasov}. But there is no obvious analogue of particle trajectories or of the method of charactristics in quantum mechanics. This explains why the proofs of Theorems \ref{T-EMQEq} and \ref{T-Hartree} require a quite 
involved algebraic setting. (See Appendix \ref{A-Hochschild} for additional insight on this algebraic structure).

\subsection{Outline of the paper}


Sections \ref{S-EM}  and  \ref{S-EqEM} are focussed on the problem of defining a notion of quantum empirical measure and deriving the equation governing its evolution. The analogue of the empirical measure in quantum mechanics 
proposed here is $\cM_N(t)$ introduced in Definition \ref{D-EMQ} (in section \ref{S-EM}). The equation governing the evolution of $\cM_N(t)$ is \eqref{EMQEq}, established in Theorem \ref{T-EMQEq} (in section \ref{S-EqEM}). 

Of critical importance in \eqref{EMQEq} is the contribution of the interaction term involving the potential $V$. This interaction term can be viewed as a twisted variant of the commutator $[V_{R(t)},R(t)]$ that appears on the right hand 
side of \eqref{Hartree}, and is defined in formula \eqref{DefC} and in Proposition \ref{P-InteractFla}. In fact, equation \eqref{EMQEq} itself is a noncommutative variant of the time-dependent Hartree equation \eqref{Hartree} (in terms 
of density operators). 

More precisely, there exists a special class (defined in \eqref{HartreeEM}) of solutions to the equation \eqref{EMQEq} satisfied by our quantum analogue of the empirical measure, for which this equation is exactly equivalent to the 
time-dependent Hartree equation \eqref{Hartree}: see Theorem \ref{T-Hartree}.

\smallskip
The main application of this new notion of quantum empirical ``measure'' is Theorem \ref{T-MFUnif}, whose proof occupies section \ref{S-EMMFh}. One important ingredient in this proof is a series of auxiliary semiclassical estimates, 
summarized in Lemma \ref{L-[E,B]}, whose proof is deferred to the last section \ref{S-SC}. Various fundamental notions and results on Weyl's quantization are recalled in Appendix \ref{A-Weyl}.


\section{A Quantum Analogue of the Notion of Empirical Measure}\lb{S-EM}


\subsection{Marginals of $N$-particle densities.}


First we recall the notion of $k$-particle marginal of a symmetric probability density $f_N$ on the $N$-particle (classical) phase-space $(\bR^d\times\bR^d)^N$.

For each $p\in[1,+\infty]$, we designate by $L^p_s((\bR^d\times\bR^d)^N)$ the set of functions $\phi\equiv\phi(x_1,\xi_1,\ldots,x_N,\xi_N)$ such that $\phi\in L^p((\bR^d\times\bR^d)^N)$ and 
$$
\phi(x_{\si(1)},\xi_{\si(1)},\ldots,x_{\si(N)},\xi_{\si(N)})=\phi(x_1,\xi_1,\ldots,x_N,\xi_N)\quad\hbox{ a.e. on }(\bR^d\times\bR^d)^N
$$
for all $\si\in\fS_N$.

For each $N>1$, let $f_N\equiv f_N(x_1,\xi_1,\ldots,x_N,\xi_N)$ be a symmetric probability density on the $N$-particle phase-space $(\bR^d\times\bR^d)^N$. In other words, 
$$
f_N\in L^1_s((\bR^d\times\bR^d)^N)\,,\qquad f_N\ge 0\quad\hbox{ a.e. on }(\bR^d\times\bR^d)^N\,,
$$
and
$$
\int_{(\bR^d\times\bR^d)^N}f_N(x_1,\xi_1,\ldots,x_N,\xi_N)\dd x_1\dd\xi_1\ldots\dd x_N\dd\xi_N=1\,.
$$

The $k$-particle marginal of $f_N$ is the symmetric probability density on the $k$-particle phase-space $(\bR^d\times\bR^d)^k$ defined by the formula
$$
\ba
f_{N:k}(x_1,\xi_1,\ldots,x_k,\xi_k)
\\
:=\int_{(\bR^d\times\bR^d)^{N-k}}f_N(x_1,\xi_1,\ldots,x_N,\xi_N)\dd x_{k+1}\dd\xi_{k+1}\ldots\dd x_N\dd\xi_N&\,.
\ea
$$

\smallskip
The analogous notion for symmetric, quantum $N$-particle density operators is defined as follows.

\begin{Def}\label{defmarg}
Let $N\ge 1$ and $F_N\in\cD_s(\fH_N)$. For each $k=1,\ldots,N$, the $k$-particle marginal of $F_N$ is the unique $F_{N:k}\in\cD_s(\fH_k)$ such that
$$
\Tr_{\fH_k}(A_kF_{N:k})=\Tr_{\fH_N}((A_k\otimes I^{\otimes(N-k)})F_N)\,,\qquad\hbox{ for all }A_k\in\cL(\fH_k)\,.
$$
(In particular $F_{N:N}=F_N$.)
\end{Def}

\subsection{A quantum notion of empirical measure.}


Let $f^{in}_N$ be a symmetric probability density on $(\bR^d\times\bR^d)^N$, and set
$$
f_N(t,x_1,\xi_1,\ldots,x_N,\xi_N):=f^{in}(T^N_{-t}(x_1,\xi_1,\ldots,x_N,\xi_N))\,,
$$
where $T^N_t$ is the flow defined by the differential system (\ref{NBodyC}). 

Specializing formula (32) in \cite{FGMouhotRicci} to $m=1$ leads to the identity
\be\lb{EMCMargi1}
f_{N:1}(t,x,\xi)\dd x\dd\xi=\int_{(\bR^d\times\bR^d)^N}\mu_{T^N_tZ_N}(\dd x\dd\xi)f^{in}_N(Z_N)\dd Z_N\,,
\ee
which holds for each symmetric probability density $f_N^{in}$ on $(\bR^d\times\bR^d)^N$. Equivalently, equation \eqref{EMCMargi1} can be recast as
\be\lb{EMCMargi1Dual}
\ba
\int_{\bR^d\times\bR^d}\phi(x,\xi)f_{N:1}(t,x,\xi)\dd x\dd\xi
\\
=\int_{(\bR^d\times\bR^d)^N}\left(\int_{\bR^d\times\bR^d}\phi(x,\xi)\mu_{T^N_tZ_N}(\dd x\dd\xi)\right)f^{in}_N(Z_N)\dd Z_N
\ea
\ee
for each test function $\phi\in L^\infty(\bR^d\times\bR^d)$.

In other words, for each $t\in\bR$, the time-dependent, measure-valued function
\be\lb{DefmN}
m_N(t;Z_N,\dd x\dd\xi):=\mu_{T^N_tZ_N}(\dd x\dd\xi)
\ee
is the integral kernel of the (unique) linear map 
\be\lb{DefMN}
M_N(t):\,L^\infty(\bR^d\times\bR^d)\to L^\infty_s((\bR^d\times\bR^d)^N)
\ee
such that
\be\lb{LinMapMN}
\int_{\bR^d\times\bR^d}(M_N(t)\phi)(Z_N)f_N^{in}(\dd Z_N):=\int_{\bR^d\times\bR^d} \phi(x,\xi)f_{N:1}(t,x,\xi)\dd x\dd\xi
\ee
for each symmetric probability density on $(\bR^d\times\bR^d)^N$. That \eqref{LinMapMN} holds for each symmetric probability density $f^{in}_N$ on $(\bR^d\times\bR^d)^N$ implies indeed that
$$
M_N(t)\phi:=\Phi_N\circ T^N_t\,,
$$
where 
$$
\Phi_N(Z_N):=\frac1N\sum_{j=1}^N\phi(x_j,\xi_j)=\int_{\bR^d\times\bR^d}\phi(x,\xi)\mu_{Z_N}(\dd x\dd\xi)\,.
$$

\bigskip
In quantum mechanics, the analogue of the differential system (\ref{NBodyC}) is the $N$-body Schr\"odinger equation 
\be\lb{NSchro}
i\hbar\d_t\Psi_N=\scrH_N\Psi_N
\ee
where $\Psi_N\equiv\Psi_N(t,x_1,\ldots,x_N)\in\bC$ is the $N$-body wave function, and where $\scrH_N$ is the quantum $N$-body Hamiltonian defined in \eqref{NHamilt}, which is recast as
\be\lb{NHamilt2}
\scrH_N:=\sum_{k=1}^N-\tfrac12\hbar^2\Dlt_{x_k}+\frac1N\sum_{1\le k<l\le N}V_{kl}\,.
\ee
For each $k,l=1,\ldots,N$, the notation $V_{kl}$ in \eqref{NHamilt2} designates the (multiplication) operator defined by 
$$
(V_{kl}\Psi_N)(x_1,\ldots,x_N):=V(x_k-x_l)\Psi_N(x_1,\ldots,x_N)\,.
$$
Since $V\in L^\infty(\bR^d)$, the operator $V_{kl}$ is bounded on $\fH_N$ for all $k,l=1,\ldots,N$, and $\scrH_N$ defines a self-adjoint operator on $\fH_N$ with domain $H^2((\bR^d)^N)$. The quantum analogue of the flow $T^N_t$ 
is the unitary group $U_N(t):=e^{it\scrH_N/\hbar}$, and, since the function $V$ is even, one easily checks that
\be\lb{PropaSym}
F^{in}_N\in\cD_s(\fH_N)\Rightarrow F_N(t):=U_N(t)^*F^{in}_NU_N(t)\in\cD_s(\fH_N)\hbox{ for all }t\in\bR\,.
\ee

\smallskip
Formula \eqref{LinMapMN} suggests that the quantum analogue of the empirical measure 
$$
\mu_{T^N_tZ_N}(\dd x\dd\xi)
$$
used in classical mechanics is the time-dependent linear map defined below. 

For each $k=1,\ldots,N$ and each $A\in\cL(\fH)$, we denote by $J_k$ the linear map from $\cL(\fH)$ to $\cL(\fH_N)$ defined by the formula
$$
J_kA:=\underbrace{I\otimes\ldots\otimes I}_{k-1\text{ terms}}\otimes A\otimes\underbrace{I\otimes\ldots\otimes I}_{N-k\text{ terms}}\,.
$$

\begin{Def}\lb{D-EMQ}
For each $N\ge 1$, set
$$
\cM^{in}_N:=\frac1N\sum_{k=1}^NJ_k\in\cL(\cL(\fH),\cL_s(\fH_N))\,.
$$
For all $t\in\bR$, we define $\cM_N(t)\in\cL(\cL(\fH),\cL_s(\fH_N))$ by the formula
$$
\cM_N(t)A:=U_N(t)(\cM^{in}_NA)U_N(t)^*\,.
$$
\end{Def}

\smallskip
With this definition, one easily arrives at the quantum analogue of \eqref{EMCMargi1Dual}.

\begin{Lem}\lb{L-EMQMargi1}
For each $F^{in}_N\!\in\!\cD_s(\fH_N)$ and all $t\in\bR$, set $F_N(t):=U_N(t)^*F^{in}_NU_N(t)$. Then one has
$$
\Tr_\fH(AF_{N:1}(t))=\Tr_{\fH_N}((\cM_N(t)A)F^{in}_N)\quad\hbox{ for all }A\in\cL(\fH)\,.
$$
\end{Lem}

\begin{proof}
Since the density operator $F^{in}_N$ is symmetric, applying \eqref{PropaSym} implies that $F_N(t)\in\cD_s(\fH_N)$ for all $t\in\bR$. Therefore
\be\lb{Margi1J}
\Tr_\fH(AF_{N:1}(t))=\Tr_{\fH_N}((J_1A)F_N(t))=\Tr_{\fH_N}((J_kA)F_N(t))
\ee
for all $k=1,\ldots,N$. Averaging both sides of \eqref{Margi1J} in $k$, we find that
$$
\ba
\Tr_\fH(\!AF_{N:1}(t)\!)\!\!=\!\Tr_{\fH_N}(\!(\cM^{in}_NA)F_N(t)\!)\!=\!\Tr_{\fH_N}(\!(\cM^{in}_NA)U_N(t)^*F^{in}_NU_N(t)\!)
\\
=\!\Tr_{\fH_N}(U_N(t)(\cM^{in}_NA)U_N(t)^*F^{in}_N)\!=\!\Tr_{\fH_N}(\!(\cM_N(t)A)F^{in}_N\!)&\,,
\ea
$$
which is the desired identity.
\end{proof}


\section{An Evolution Equation for $\cM_N(t)$}\lb{S-EqEM}


In the present section, we seek to establish a quantum analogue of Proposition \ref{P-WSolVlasov} for the time dependent linear map $t\mapsto\cM_N(t)$. At variance with the proof of Proposition \ref{P-WSolVlasov}, which is based on
the method of characteristics for the transport equation, and of which there is no quantum analogue as mentioned above, our approach to this problem is based solely on the first equation in the BBGKY hierarchy.

\subsection{The first BBGKY equation}\lb{SS-BBGKY1}

The time dependent density $t\mapsto F_N(t)$ given in terms of the initial quantum density $F^{in}_N$ by the formula 
$$
F_N(t):=U_N(t)^*F^{in}_NU_N(t)\,,
$$
is the solution of the Cauchy problem for the Heisenberg equation \eqref{NHeisen}.

The BBGKY hierarchy is the sequence of differential equations for the marginals $F_{N:k}$ (for $k=1,\ldots,N$) deduced from (\ref{NHeisen}). Because of the pairwise particle interaction, the equation for $F_{N:k}$ always involves 
$F_{N:k+1}$ and is never in closed form for $1\le k\le N-1$. 

Nevertheless, only the first equation in the BBGKY hierarchy is needed in the sequel. It is obtained as follows: let $A\in\cL(\fH)$ satisfy $[\Dlt,A]\in\cL(\fH)$. Multiplying both sides of (\ref{NHeisen}) by $J_1A$ and taking the trace shows 
that
\be\lb{idtF1A}
\ba
i\hbar\d_t\Tr_\fH(AF_{N:1}(t))=&i\hbar\d_t\Tr_{\fH_N}((J_1A)F_N(t))
\\
=&\Tr_{\fH_N}((J_1A)[\scrH_N,F_N(t)])
\\
=&-\Tr_{\fH_N}([\scrH_N,J_1A]F_N(t))\,,
\ea
\ee
where the first equality above follows from the first identity in \eqref{Margi1J}. Next one has
$$
\ba
\left[\scrH_N,J_1A\right]&=\sum_{k=1}^N\left[-\tfrac12\hbar^2\Dlt_{x_k},J_1A\right]+\frac1N\sum_{1\le k<l\le N}[V_{kl},J_1A]
\\
&=J_1\left(\left[-\tfrac12\hbar^2\Dlt,A\right]\right)+\frac1N\sum_{l=2}^N[V_{1l},J_1A]\,.
\ea
$$
Using the first identity in \eqref{Margi1J} with $\left[-\tfrac12\hbar^2\Dlt,A\right]$ in the place of $A$ shows that
\be\lb{rhsidtF1A}
\ba
\Tr_{\fH_N}([\scrH_N,J_1A]F_N(t))=&\Tr_{\fH}\left(\left[-\tfrac12\hbar^2\Dlt,A\right]F_{N:1}(t)\right)
\\
&+\frac1N\sum_{l=2}^N\Tr_{\fH_N}([V_{1l},J_1A]F_N(t))\,.
\ea
\ee
If $\si$ is the transposition exchanging $2$ and $l=3,\ldots,N$, then
$$
U_\si[V_{1l},J_1A]U_\si^*=[V_{12},J_1A]\,,
$$
and since $F_N(t)$ is symmetric, for each $l=3,\ldots,N$, one has
$$
\ba
\Tr_{\fH_N}([V_{1l},J_1A]F_N(t))=&\Tr_{\fH_N}(U_\si[V_{1l},J_1A]U_\si^*F_N(t))
\\
=&\Tr_{\fH_N}([V_{12},J_1A]F_N(t))
\\
=&\Tr_{\fH_N}\left(\left([V_{12},A\otimes I]\otimes I^{\otimes(N-2)}\right)F_N(t)\right)
\\
=&\Tr_{\fH_2}([V_{12},A\otimes I]F_{N:2}(t))\,.
\ea
$$
Hence
\be\lb{SymInteract}
\frac1N\sum_{l=2}^N\Tr_{\fH_N}([V_{1l},J_1A]F_N(t))=\frac{N-1}N\Tr_{\fH_2}([V_{12},A\otimes I]F_{N:2}(t))\,.
\ee

Substituting the second term in the right hand side of \eqref{rhsidtF1A} with the right hand side of \eqref{SymInteract} shows that \eqref{idtF1A} can be put in the form
\be\lb{BBGKY1}
\ba
i\hbar\d_t\Tr_\fH(AF_{N:1}(t))+\Tr_{\fH}\left(\left[-\tfrac12\hbar^2\Dlt,A\right]F_{N:1}(t)\right)
\\
=-\frac{N-1}N\Tr_{\fH_2}([V_{12},A\otimes I]F_{N:2}(t))&\,,
\ea
\ee
for all $A\in\cL(\fH)$ such that $[\Dlt,A]\in\cL(\fH)$, which is the first BBGKY equation\footnote{More precisely, it is the weak formulation of the first BBGKY equation.}.

\medskip
Using Lemma \ref{L-EMQMargi1}, we easily express the left-hand side of \eqref{BBGKY1} in terms of $\cM_N(t)$:
$$
\ba
i\hbar\d_t\Tr_\fH(AF_{N:1}(t))+\Tr_{\fH}\left(\left[-\tfrac12\hbar^2\Dlt,A\right]F_{N:1}(t)\right)
\\
=i\hbar\d_t\Tr_{\fH_N}((\cM_N(t)A)F^{in}_N)+\Tr_{\fH_N}\left(\left(\cM_N(t)\left[-\tfrac12\hbar^2\Dlt,A\right]\right)F^{in}_N\right)&\,.
\ea
$$
For each $\L\in\cL(\cL(\fH),\cL_s(\fH_N))$ and each (possibly unbounded) operator $D$ on $\fH$, we henceforth denote\footnote{By analogy with the notation for the co-adjoint representation of a Lie algebra.} by $\ad^*(D)\L$ the linear map 
defined by
\be\lb{NotLR(D)}
(\ad^*(D)\L)A:=-\L[D,A]
\ee
for all $A\in\cL(\fH)$ such that $[D,A]\in\cL(\fH)$. With this notation, \eqref{BBGKY1} becomes
\be\lb{BBGKY1b}
\ba
\Tr_{\fH_N}\left(\left(\left(i\hbar\d_t\cM_N(t)-\ad^*(-\tfrac12\hbar^2\Dlt)\cM_N(t)\right)A\right)F^{in}_N\right)
\\
=-\frac{N-1}N\Tr_{\fH_2}([V_{12},A\otimes I]F_{N:2}(t))&\,.
\ea
\ee

\subsection{The interaction term}\lb{SS-Interact}

The present section is focussed on the problem of expressing the right hand side of \eqref{BBGKY1b} in terms of $\cM_N(t)$, which is the most critical part of our analysis. Since the right hand side of \eqref{BBGKY1b} involves $F_{N:2}$
and cannot be expressed exclusively in terms of $F_{N:1}$, it is not a priori obvious that there is an equation in closed form governing the evolution of $\cM_N(t)$.

\smallskip
First, we introduce some additional notation. For each $\om\in\bR^d$, we denote by $E_\om\in\cL(\fH)$ the multiplication operator defined by the formula
$$
E_\om\psi(x):=e^{i\om\cdot x}\psi(x)\,.
$$
The family of operators $E_\om$ obviously satisfies
$$
E^*_\om=E_{-\om}=E_\om^{-1}\,,
$$
For each $\cL\in\cL(\cL(\fH),\cL_s(\fH_N))$ and each $B\in\cL(\fH)$, denote by $\L(\bu B)$ and $\L(B\bu)$ the elements of $\cL\in\cL(\cL(\fH),\cL_s(\fH_N))$ defined by 
$$
\L(\bu B):\;A\mapsto\L(AB)\quad\hbox{ and }\quad\L(B\bu):\;A\mapsto\L(BA)
$$

\smallskip
\begin{Def}\lb{D-DefC}
For each integer $N\ge 2$, set
\be\lb{DefEN}
\ba
\cE_N:=\{(\L_1,\L_2)\in\cL(\cL(\fH),\cL(\fH_N))^2\hbox{ s.t. for all }F\in\cL^1(\fH_N)\hbox{ the map}&
\\
\om\mapsto\Tr_{\fH_N}((\L_1E^*_\om)(\L_2E_\om)F)\hbox{is continuous on }\bR&\}\,.
\ea
\ee
For each $S\in\cS'(\bR)$ whose Fourier transform $\hat S$ is a bounded measure on $\bR$, and for each $\L_1,\L_2\in\cL(\cL(\fH),\cL(\fH_N))$ such that
$$
(\L_1,\L_2(\bu A))\hbox{ and }(\L_2(A\bu),\L_1)\in\cE_N\quad\hbox{ for all }A\in\cL(\fH)\,,
$$
let $\cC[S,\L_1,\L_2]$ be the element of $\cL(\cL(\fH),\cL(\fH_N))$ defined by the formula
\be\lb{DefC}
\cC[S,\L_1,\L_2]A:=\int_{\bR}\left((\L_1E^*_\om)(\L_2(E_\om A))-(\L_2(AE_\om))(\L_1E^*_\om)\right)\hat S(\dd\om)
\ee
for all $A\in\cL(\fH)$. The linear map $\cC[S,\L_1,\L_2]$ satisfies
\be\lb{BoundCSL1L2}
\|\cC[S,\L_1,\L_2]\|\le 2\|\L_1\|\|\L_2\|\|\hat S\|_{TV}\,.
\ee
\end{Def}

The integral defining $\cC[S,\L_1,\L_2]A$ on the right hand side of \eqref{DefC} is to be understood in the ultraweak sense{\footnote{Let $\fH$ be a separable Hilbert space. The ultraweak topology is the topology defined on $\cL(\fH)$ 
by the family of seminorms $A\mapsto|\Tr_\fH(AF)|$ as $F$ runs through $\cL^1(\fH)$. Let $m$ be a bounded, complex-valued Borel measure on $\bR^d$, and let $f:\,\bR^d\to\cL(\fH)$ be ultraweakly continuous and such that
$$
\sup_{\om\in\bR^d}\|f(\om)\|<\infty\,.
$$ 
Then the linear functional 
$$
\cL^1(\fH)\ni F\mapsto\la L_{f,m},F\ra:=\int_{\bR^d}\Tr_{\fH}(f(x)F)m(\dd\om)\in\bC
$$
is continuous with norm 
$$
\|L_{f,m}\|\le\sup_{\om\in\bR^d}\|f(\om)\|\|m\|_{TV}\,.
$$
Hence there exists a unique $\Phi_m\in\cL(\fH)$ such that $\la L_{f,m},F\ra=\Tr_{\fH}(\Phi_mF)$. The operator $\Phi_m$ is the ultraweak integral of $fm$, denoted
$$
\int_{\bR^d}f(\om)m(\dd\om):=\Phi_m\,.
$$}}. 
Indeed, for each $A\in\cL(\fH)$, the function
$$
\bR^d\ni\om\mapsto (\L_1E^*_\om)(\L_2(E_\om A))-(\L_2(AE_\om))(\L_1E^*_\om)\in\cL(\fH_N)
$$
is ultraweakly continuous since $(\L_1,\L_2(\bu A))$ and $(\L_2(A\bu),\L_1)\in\cE_N$. Moreover
$$
\|(\L_1E^*_\om)(\L_2(E_\om A))-(\L_2(AE_\om))(\L_1E^*_\om)\|\le 2\|\L_1\|\|\L_2\|\|A\|\,.
$$
Hence the integral on the right hand side of \eqref{DefC} is well defined in the ultraweak sense, and satisfies
$$
\|\cC[S,\L_1,\L_2]A\|\le 2\|\L_1\|\|\L_2\|\|A\|\int_{\bR^d}|\hat S|(\dd\om)=2\|\L_1\|\|\L_2\|\|A\|\|\hat S\|_{TV}
$$
for all $A\in\cL(\fH)$, which implies \eqref{BoundCSL1L2}.

\medskip
The main result in this section is the following proposition.

\begin{Prop}\lb{P-InteractFla}
Let $V$ be an even, real-valued element of $\cS'(\bR^d)$ whose Fourier transform $\hat V$ is a bounded measure. Let $F_N$ be the solution of the Cauchy problem for the Heisenberg equation \eqref{NHeisen} 
with initial condition $F^{in}_N$. Then, for each $t\in\bR$, each integer $N>1$ and each $A\in\cL(\fH)$, one has
$$
\frac{N-1}N\Tr_{\fH_2}\left([V_{12},A\otimes I]F_{N:2}(t)\right)=\Tr_{\fH_N}\left((\cC[V,\cM_N(t),\cM_N(t)]A)F^{in}_N\right)
$$
where $\cC[V,\cM_N(t),\cM_N(t)]A$ is the element of $\cL_s(\fH_N)$ defined in \eqref{DefC}.
\end{Prop}

\smallskip
\begin{proof}
The proof of this proposition is based on a rather involved computation and will be decomposed in several steps.

\smallskip
\noindent
\textit{Step 1.} The next two lemmas are elementary, but of crucial importance in the sequel.

\begin{Lem}\lb{L-ComMu}
For each integer $N\ge 1$ and each $A,B\in\cL(\fH)$, one has
$$
[J_kA,J_lB]=\de_{kl}J_k[A,B]\,,\qquad k,l=1,\ldots,N\,,
$$
and
$$
[\cM^{in}_NA,\cM^{in}_NB]=\frac1N\cM^{in}_N[A,B]\,.
$$
\end{Lem}

\begin{proof}[Proof of Lemma \ref{L-ComMu}]
The first formula being obvious, one has
$$
\ba
\left[\cM^{in}_NA,\cM^{in}_NB\right]=&\frac1{N^2}\sum_{k,l=1}^N[J_kA,J_lB]
\\
=&\frac1{N^2}\sum_{k,l=1}^N\de_{kl}J_k[A,B]=\frac1{N^2}\sum_{k=1}^NJ_k[A,B]=\frac1N\cM^{in}_N[A,B]\,.
\ea
$$
\end{proof}

\begin{Lem}\lb{L-MuF2} For each $N\ge 1$, each $A,B\in\cL(\fH)$ and each $F_N\in\cL^1_s(\fH_N)$, one has
$$
\Tr_{\fH_N}((\cM^{in}_NA)(\cM^{in}_NB)F_N)=\frac{N-1}N\Tr_{\fH_2}((A\otimes B)F_{N:2})+\frac1N\Tr_{\fH}(ABF_{N:1})\,.
$$
\end{Lem}

\begin{proof}[Proof of Lemma \ref{L-MuF2}]
Since $F_N\in\cL^1_s(\fH_N)$, one has
$$
\Tr_{\fH_N}((J_kA)(J_lB)F_N)=\Tr_{\fH_N}((J_1A)(J_2B)F_N)=\Tr_{\fH_2}((A\otimes B)F_{N:2})
$$
for all $A,B\in\cL(\fH)$, if $1\le k\not=l\le N$, while
$$
\Tr_{\fH_N}((J_kA)(J_kB)F_N)=\Tr_{\fH_N}(J_k(AB)F_N)=\Tr_{\fH_1}(ABF_{N:1})
$$
for all $k=1,\ldots,N$. Hence
$$
\ba
\Tr_{\fH_N}((\cM^{in}_NA)(\cM^{in}_NB)F_N)
\\
=\frac1{N^2}\sum_{1\le k\not=l\le N}\Tr_{\fH_N}((J_kA)(J_lB)F_N)+\frac1{N^2}\sum_{k=1}^N\Tr_{\fH_N}((J_kA)(J_kB)F_N)&
\\
=\frac{N(N-1)}{N^2}\Tr_{\fH_2}((A\otimes B)F_{N:2})+\frac1N\Tr_{\fH_1}(ABF_{N:1})&\,.
\ea
$$
\end{proof}

\smallskip
\noindent
\textit{Step 2.} For each $\phi_1,\phi_2,\psi_1,\psi_2\in\fH$, one has
$$
\Tr_{\fH_2}((E_\om\otimes E^*_\om)|\phi_1\otimes\phi_2\ra\la\psi_1\otimes\psi_2|)|=\widehat{\overline{\psi_1}\phi_1}(-\om)\widehat{\overline{\psi_2}\phi_2}(\om)
$$
so that, by density in $\cL^1(\fH_2)$ of the linear span of rank-one projections of the form $|\phi_1\otimes\phi_2\ra\la\psi_1\otimes\psi_2|$ and dominated convergence, the function 
$$
\bR^d\ni\om\mapsto E^*_\om\otimes E_\om\in\cL(\fH_2)
$$ 
is ultraweakly continuous. Besides, by definition of $V_{12}$,
$$
(V_{12}\phi_1\otimes\phi_2)(x,y)=V(x-y)\phi_1(x)\phi_2(y)\,,
$$
so that
\be\lb{TrVrk1}
\ba
\la\psi_1\otimes\psi_2|V_{12}|\phi_1\otimes\phi_2\ra=\iint_{\bR^d\times\bR^d}V(x-y)\overline{\psi_1(x)}\phi_1(x)\overline{\psi_2(y)}\phi_2(y)\dd x\dd y
\\
=\int_{\bR^d}\overline{\psi_2(y)}\phi_2(y)\left(\tfrac1{(2\pi)^d}\int_{\bR^d}\widehat{\overline{\psi_1}\phi_1}(-\om)e^{-i\om\cdot y}\hat V(\dd\om)\right)\dd y
\\
=\tfrac1{(2\pi)^d}\int_{\bR^d}\widehat{\overline{\psi_1}\phi_1}(-\om)\widehat{\overline{\psi_2}\phi_2}(\om)\hat V(\dd\om)&\,.
\ea
\ee
Since $\|E^*_\om\otimes E_\om\|\le 1$, we conclude from \eqref{TrVrk1} that
$$
V_{12}=\tfrac1{(2\pi)^d}\int_{\bR^d}E_\om\otimes E^*_\om\hat V(\dd\om)
$$
in the ultraweak sense. In other words
$$
\Tr_{\fH_2}(V_{12}G)=\tfrac1{(2\pi)^d}\int_{\bR^d}\Tr_{\fH_2}(E_\om\otimes E^*_\om G)\hat V(\dd\om)
$$
for all $G\in\cL^1(\fH_2)$. Setting successively $G=(A\otimes I)F_{N:2}(t)$ and $G=F_{N:2}(t)(A\otimes I)$ in this identity, we conclude that
\be\lb{Ident0}
\ba
\Tr_{\fH_2}([V_{12},A\otimes I]F_{N:2}(t))
\\
=\tfrac1{(2\pi)^d}\int_{\bR^d}\hat V(\om)\Tr_{\fH_2}(([E_\om,A]\otimes E^*_\om)F_{N:2}(t))\dd\om&\,.
\ea
\ee

\smallskip
\noindent
\textit{Step 3.} At this point, we use Lemma \ref{L-MuF2} and obtain the identity
\be\lb{Ident1}
\ba
\tfrac{N-1}N\Tr_{\fH_2}(([E_\om,A]\otimes E^*_\om)F_{N:2}(t))
\\
=\Tr_{\fH_N}((\cM^{in}_N[E_\om,A])(\cM^{in}_NE^*_\om)F_N(t))-\tfrac1N\Tr_{\fH}([E_\om,A]E^*_\om F_{N:1}(t))
\ea
\ee
for each $\om\in\bR$ and each $t\in\bR$.

Next observe that
$$
[E_\om,A]E^*_\om=E_\om AE^*_\om-A=[E_\om A,E^*_\om]\,,
$$
since $E^*_\om=E_\om^{-1}$. Hence, using the second identity in Lemma \ref{L-ComMu} shows that
$$
\ba
\tfrac1N\Tr_{\fH}([E_\om,A]E^*_\om F_{N:1}(t))&=\tfrac1N\Tr_{\fH}([E_\om A,E^*_\om]F_{N:1}(t))
\\
&=\tfrac1N\Tr_{\fH_N}((\cM^{in}_N[E_\om A,E^*_\om])F_N(t))
\\
&=\Tr_{\fH_N}([\cM^{in}_N(E_\om A),\cM^{in}_NE^*_\om]F_N(t))\,.
\ea
$$

Substituting the expression for $\tfrac1N\Tr_{\fH}([E_\om,A]E^*_\om F_{N:1}(t))$ so obtained in the right hand side of (\ref{Ident1}), we see that
\be\lb{InteractIntegrand0}
\ba
\tfrac{N-1}N\Tr_{\fH_2}(([E_\om,A]\otimes E^*_\om)F_{N:2}(t))
\\
=\Tr_{\fH_N}\left(\left((\cM^{in}_N[E_\om,A])(\cM^{in}_NE^*_\om)-[\cM^{in}_N(E_\om A),\cM^{in}_NE^*_\om]\right)F_N(t)\right)
\\
=\Tr_{\fH_N}\left(\left((\cM^{in}_NE^*_\om)\cM^{in}_N(E_\om A)-\cM^{in}_N(AE_\om)(\cM^{in}_NE^*_\om)\right)F_N(t)\right)&\,.
\ea
\ee

\smallskip
\noindent
\textit{Step 4.} Since $F_N(t)=U_N(t)^*F_N^{in}U_N(t)$, one has
$$
\ba
\Tr_{\fH_N}\left(\left((\cM^{in}_NE^*_\om)\cM^{in}_N(E_\om A)-\cM^{in}_N(AE_\om)(\cM^{in}_NE^*_\om)\right)F_N(t)\right)
\\
=\Tr_{\fH_N}\left(U_N(t)\left((\cM^{in}_NE^*_\om)\cM^{in}_N(E_\om A)-\cM^{in}_N(AE_\om)(\cM^{in}_NE^*_\om)\right)U_N(t)^*F_N^{in}\right)
\\
=\Tr_{\fH_N}\left(\left((\cM_N(t)E^*_\om)\cM_N(t)(E_\om A)-\cM_N(t)(AE_\om)(\cM_N(t)E^*_\om)\right)F_N^{in}\right)
\ea
$$
by definition of $\cM_N(t)$ (see Definition \ref{D-EMQ}). Hence we deduce from \eqref{InteractIntegrand0} that
\be\lb{InteractIntegrand}
\ba
\frac{N-1}N\Tr_{\fH_2}(([E_\om,A]\otimes E^*_\om)F_{N:2}(t))
\\
=\Tr_{\fH_N}\left(\left((\cM_N(t)E^*_\om)\cM_N(t)(E_\om A)-\cM_N(t)(AE_\om)(\cM_N(t)E^*_\om)\right)F_N^{in}\right)
\ea
\ee

Arguing as in Step 2, observe that formula \eqref{InteractIntegrand} implies that, for each $A\in\cL(\fH)$, one has both
$$
(\cM_N(t),\cM_N(t)(\bu A))\hbox{ and }(\cM_N(t)(A\bu),\cM_N(t))\in\cE_N
$$
for all $t\in\bR$. Then, \eqref{Ident0} and \eqref{InteractIntegrand} imply that
$$
\ba
\frac{N-1}{N}\Tr_{\fH_2}([V_{12},A\otimes I]F_{N:2}(t))
\\
=\tfrac1{(2\pi)^d}\int_{\bR^d}\Tr_{\fH_N}\left((\cM_N(t)E^*_\om)(\cM_N(t)(E_\om A))F_N^{in}\right)\hat V(\dd\om)
\\
-\tfrac1{(2\pi)^d}\int_{\bR^d}\Tr_{\fH_N}\left((\cM_N(t)(AE_\om))(\cM_N(t)E^*_\om)F_N^{in}\right)\hat V(\dd\om)
\\
=\Tr_{\fH_N}((\cC[V,\cM_N(t),\cM_N(t)]A)F_N^{in})&\,,
\ea
$$
which is precisely the desired expression of the interaction term.
\end{proof}

\subsection{Writing an equation for $\cM_N(t)$}\lb{SS-EMQEq}

With \eqref{BBGKY1b} and Proposition \ref{P-InteractFla}, we see that $\cM_N(t)$ satisfies
$$
\ba
\Tr_{\fH_N}\left(\left(\left(i\hbar\d_t\cM_N(t)-\ad^*(-\tfrac12\hbar^2\Dlt)\cM_N(t)\right)A\right)F^{in}_N\right)
\\
=\Tr_{\fH_N}\left(\left(\cC[V,\cM_N(t),\cM_N(t)]A\right)F^{in}_N\right)&
\ea
$$
for all $A\in\cL(\fH)$ such that $[\Dlt,A]\in\cL(\fH)$, and all $F^{in}_N\in\cL_s(\fH_N)$. 

\smallskip
Our first main result in this paper is the following theorem.

\begin{Thm}\lb{T-EMQEq}
Let $V\in C_b(\bR^d)$ be an even, real-valued function whose Fourier transform $\hat V$ is a bounded measure. Let
$$
U_N(t)=e^{it\scrH_N/\hbar}\,,\quad\hbox{ with }\scrH_N:=\sum_{k=1}^N-\tfrac12\hbar^2\Dlt_{x_k}+\frac1N\sum_{1\le k<l\le N}V(x_k-x_l)\,,
$$
and let $\cM_N(t)$ be the element of $\cL(\cL(\fH),\cL_s(\fH_N))$ in Definition \ref{D-EMQ}. Then
\be\lb{EMQEq}
i\hbar\d_t\cM_N(t)=\ad^*(-\tfrac12\hbar^2\Dlt)\cM_N(t)-\cC[V,\cM_N(t),\cM_N(t)]\,.
\ee
\end{Thm}

\smallskip
\noindent
\textbf{Remark.} Unfortunately the terms
$$
\ad^*(-\tfrac12\hbar^2\Dlt)\cM_N(t)\quad\hbox{ and }\quad\cC[V,\cM_N(t),\cM_N(t)]
$$
on the right hand side of \eqref{EMQEq} seem to be of a very different nature. This is in sharp contrast with the Hartree equation, where the kinetic energy $-\tfrac12\hbar^2\Dlt$ and the interaction potential energy $V_{R(t)}$ defined in 
\eqref{MFPot} appear together on an equal footing in right hand side of \eqref{Hartree}. However, both terms appearing on the right hand side of \eqref{EMQEq} can be formally assembled in a single expression which is reminiscent of 
the commutator appearing on the right hand side of \eqref{Hartree}. See Appendix \ref{A-Hochschild} for a complete discussion of this point.

\begin{proof}
Setting
$$
S_N:=i\hbar\d_t(\cM_N(t)A)-\left(\ad^*(-\tfrac12\hbar^2\Dlt)\cM_N(t)\right)A+\cC[V,\cM_N(t),\cM_N(t)]A\,,
$$
one has 
$$
S_N\in\cL_s(\fH_N)\quad\hbox{ and }\Tr_{\fH_N}(S_NF^{in}_N)=0
$$
for all $F^{in}_N\in\cD_s(\fH_N)$. 

\begin{Lem}
Any $G_N\in\cL^1_s(\fH_N)$ can be decomposed as
\be\lb{Decomp4}
G_N=\sum_{k=1}^4\l_kG_N^k\,,\quad\hbox{ with }G_N^k\in\cD_s(\fH_N)\hbox{ and }\l_k\in\bC\,.
\ee
\end{Lem}

Therefore
$$
\Tr_{\fH_N}(S_NG_N)=0\quad\hbox{ for all }G_N\in\cL^1_s(\fH_N)\,.
$$
Pick $G_N$ of the form
$$
G_N:=S^*_N|\Psi_N\ra\la\Psi_N|\,,
$$
where
$$
\Psi_N\in\fH_N\,,\quad\hbox{ such that }\|\Psi_N\|_{\fH_N}=1\,,\quad \hbox{ and }U_\si\Psi_N=\Psi_N\hbox{ for all }\si\in\fS_N\,.
$$
Thus
$$
0=\Tr_{\fH_N}(S_NG_N)=\la\Psi_N|S_NS^*_N|\Psi_N\ra=\|S^*_N\Psi_N\|^2_{\fH_N}\,,
$$
i.e.
$$
S^*_N\Psi_N=0\quad\hbox{ for all }\Psi_N\in\fH_N\hbox{ such that }U_\si\Psi_N=\Psi_N\hbox{ for all }\si\in\fS_N\,.
$$
For each $\Phi_N\in\fH_N$, one has
$$
S^*_N\Phi_N=S^*_N\tilde\Phi_N\,,\qquad\hbox{ with }\tilde\Phi_N=\frac1{N!}\sum_{\tau\in\fS_N}U_\tau\Phi_N\,.
$$
By construction, $U_\si\tilde\Phi_N=\tilde\Phi_N$ for all $\si\in\fS_N$, so that $S^*_N\Phi_N=0$ for all $\Phi_N\in\fH_N$. Hence $S^*_N=0$, or equivalently $S_N=0$. Since this holds for all $A\in\cL(\fH)$ such that $[\Dlt,A]\in\fH$, 
we conclude that \eqref{EMQEq} holds.
\end{proof}

\begin{proof}[Proof of the lemma]
Indeed, write 
$$
G_N=\Re(G_N)+i\Im(G_N)\,,\quad\hbox{ with }\left\{\ba\Re(G_N):=&\tfrac12(G_N+G_N^*)\,,\\ \Im(G_N):=&\tfrac1{2i}(G_N-G_N^*)\,.\ea\right.
$$
One has
$$
\Re(G_N)=\Re(G_N)^*\hbox{ and }\Im(G_N)=\Im(G_N)^*\in\cL^1_s(\fH_N)\,,
$$
so that $\Re(G_N)$ and $\Im(G_N)$ have spectral decompositions of the form
$$
\Re(G_N)=\sum_{n\ge 0}\a_n\bp_n\,,\quad\Im(G_N)=\sum_{n\ge 0}\b_n\bq_n\,,
$$
where 
$$
\left\{\ba\bp_n^2&=\bp_n=\bp_n^*=U_\si\bp_nU_\si^*\,,\\ \bq_n^2&=\bq_n=\bq_n^*=U_\si\bq_nU_\si^*\,,\ea\right.\quad\hbox{ and }m\not=n\Rightarrow\left\{\ba\Tr_{\fH_N}(\bp_n\bp_m)=0\,,\\ \Tr_{\fH_N}(\bq_n\bq_m)=0\,,\ea\right.
$$
for all $m,n\ge 0$ and $\si\in\fS_N$, and
$$
\sum_{n\ge 0}|\a_n|\Tr_{\fH_N}(\bp_n)<\infty\,,\quad\sum_{n\ge 0}|\b_n|\Tr_{\fH_N}(\bq_n)<\infty\,.
$$
The decomposition \eqref{Decomp4} is obtained with
$$
\ba
\l_1G_N^1&=\sum_{n\ge 0}\a^+_n\bp_n\,,&&\quad\l_2G_N^2=-\sum_{n\ge 0}\a^-_n\bp_n\,,
\\
\l_3G_N^3&=\sum_{n\ge 0}i\b^+_n\bq_n\,,&&\quad\l_4G_N^4=-\sum_{n\ge 0}i\b^-_n\bq_n\,,
\ea
$$
setting
$$
\l_k=\Tr_{\fH_N}(\l_kG^k_n)\,,\quad k=1,\ldots,4\,.
$$
\end{proof}

\subsection{Hartree's equation is a special case of equation \eqref{EMQEq}}

In order to complete the analogy between Proposition \ref{P-WSolVlasov} and Theorem \ref{T-EMQEq}, we now explain the connection between Hartree's equation (the quantum analogue of Vlasov's equation (\ref{Vlasov})) and the evolution 
equation \eqref{EMQEq} satisfied by $t\mapsto\cM_N(t)$.

In the classical case, $t\mapsto\mu_{T^N_tZ_N}(\dd x\dd\xi):=m(t,\dd x\dd\xi;Z_N)$ is a measured-valued function of $t$, parametrized by $Z_N$. This measured-valued function of time $m$ is a solution of the Vlasov equation \eqref{Vlasov} 
for all $Z_N\in(\bR^d\times\bR^d)^N$: the partial derivatives in (\ref{Vlasov}) act on the variables $t,x,\xi$ in $m$, and not on $Z_N$. Any classical solution $f\equiv f(t,x,\xi)$ of the Vlasov equation can be viewed as special case of $m$ of the 
form $m(t,\dd x\dd\xi;Z_N)=f(t,x,\xi)\dd x\dd\xi$; most importantly, such an $m$ is constant in $Z_N$. 

The quantum analogue of such an $m$ is a time-dependent element $t\mapsto\scrR(t)$ of $\cL(\cL(\fH),\cL_s(\fH_N))$ of the form
\be\lb{HartreeEM}
\scrR(t):\,A\mapsto\Tr_{\fH}(R(t)A)I_{\fH_N}\,,
\ee
assuming that $R(t)\in\cD(\fH)$ for all $t\in\bR$.

\begin{Thm}\lb{T-Hartree}
The time-dependent element $t\mapsto\scrR(t)$ of $\cL(\cL(\fH),\cL_s(\fH_N))$ defined by (\ref{HartreeEM}) is a solution to the evolution equation (\ref{EMQEq}) if and only if $t\mapsto R(t)$ is a solution of the Hartree equation 
\be\lb{HartreeEq}
i\hbar\d_tR(t)=[-\tfrac12\hbar^2\Dlt+V_{R(t)},R(t)]\,.
\ee
Here $V_{R(t)}\in\cL(\fH)$ designates the operator defined by
$$
(V_{R(t)}\phi)(x)=\phi(x)\int_{\bR^d}V(x-z)r(t,z,z)\dd z\,,
$$
denoting by $r(t,x,y)$ the integral kernel of the trace-class operator $R(t)$.
\end{Thm}

\smallskip
The key observation in the proof of Theorem \ref{T-Hartree} is summarized in the next lemma.

\begin{Lem}\lb{L-LRcP}
Let $R\in\cD(\fH)$ and define $\scrR\in\cL(\cL(\fH),\cL_s(\fH_N))$ by the formula
$$
\scrR A:=\Tr_\fH(RA)I_{\fH_N}\,,\qquad\hbox{ for all }A\in\cL(\fH)\,.
$$
Then, for each $\L\in\cL(\cL^1(\fH_N),\cL^1(\fH))$ and each $A\in\cL(\fH)$, one has
$$
(\scrR,\L^*(\bu A))\hbox{ and }(\L^*(A\bu),\scrR)\in\cE_N
$$
(where $\L^*\in\cL(\cL(\fH),\cL(\fH_N))$ is the adjoint of $\L$), and
$$
\cC[V,\scrR,\L^*]=-\ad^*(V_R)\L^*
$$
for all even, real-valued $V\in C_b(\bR^d)$ whose Fourier transform $\hat V$ is a bounded measure.
\end{Lem}

\begin{proof}[Proof of the lemma]
For all $G\in\cL^1(\fH_N)$, one has
$$
\ba
\Tr_{\fH_N}((\scrR E^*_\om)(\L^*(E_\om A))G)=&\Tr_{\fH}(RE^*_\om)\Tr_{\fH_N}((\L^*(E_\om A))G)
\\
=&\Tr_{\fH}(RE^*_\om)\Tr_{\fH}(E_\om A(\L G))
\ea
$$
and likewise
$$
\Tr_{\fH_N}((\L^*(AE_\om))(\scrR E^*_\om)G)=\Tr_{\fH}(RE^*_\om)\Tr_{\fH}(E_\om(\L G)A)\,.
$$
Hence both functions
$$
\om\mapsto\Tr_{\fH_N}((\scrR E^*_\om)(\L^*(E_\om A))G)\hbox{ and }\om\mapsto\Tr_{\fH_N}((\L^*(AE_\om))(\scrR E^*_\om)G)
$$
are continuous on $\bR$, so that
$$
(\scrR,\L^*(\bu A))\hbox{ and }(\L^*(A\bu),\scrR)\in\cE_N\,.
$$

Since $V$ is even, the bounded measure $\hat V$ is invariant under the transformation $\om\mapsto-\om$, and therefore
$$
\ba
\cC[V,\scrR,\L^*]A
\\
=\tfrac1{(2\pi)^d}\int_{\bR^d}\left(\Tr(RE^*_\om)\L^*(E_\om A)-\L^*(AE_\om)\Tr(RE^*_\om)\right)\hat V(\dd\om)
\\
=\tfrac1{(2\pi)^d}\int_{\bR^d}\hat V(\om)\Tr(RE^*_\om)\L^*([E_\om,A])\dd\om
\\
=\L^*\left[\tfrac1{(2\pi)^d}\int_{\bR^d}\hat V(\om)\Tr(RE^*_\om)E_\om\dd\om,A\right]&\,.
\ea
$$
Since
$$
\Tr(RE^*_\om)=\hat\rho(\om)\,,\quad\hbox{ where }\rho(x):=r(x,x)
$$
(where $r\equiv r(x,y)$ is the integral kernel of the trace-class operator $R$), one has
$$
\tfrac1{(2\pi)^d}\int_{\bR^d}\hat V(\om)\Tr(RE^*_\om)E_\om\dd\om=\tfrac1{(2\pi)^d}\int_{\bR^d}\hat V(\om)\hat\rho(\om)E_\om\dd\om=V\star\rho=V_R\,.
$$
Therefore
$$
\cC[V,\scrR,\L^*]A=\L^*\left[V_R,A\right]=-(\ad^*(V_R)\L^*)A\,.
$$
\end{proof}

\begin{proof}[Proof of the theorem] First observe that $\scrR(t)$ is the adjoint of linear map
$$
\cL^1(\fH_N)\ni G\mapsto(\Tr_{\fH_N}(G))R(t)\in\cL^1(\fH)\,,
$$
which is obviously continuous, since
$$
\|(\Tr_{\fH_N}(G))R(t)\|_{\cL^1(\fH)}\le\|G\|_{\cL^1}\|R(t)\|_{\cL^1(\fH)}=\|G\|_{\cL^1}\,.
$$
By Lemma \ref{L-LRcP}
$$
\cC[V,\scrR(t),\scrR(t)]=-\ad^*(V_{R(t)})\scrR(t)
$$
so that, for each $A\in\cL(\fH)$
$$
\ba
\cC[V,\scrR(t),\scrR(t)]A=(-\ad^*(V_{R(t)})\scrR(t))A=\scrR(t)[V_{R(t)},A]
\\
=-\Tr_{\fH}(R(t)[V_{R(t)},A])I_{\fH_N}=\Tr_{\fH}([V_{R(t)},R(t)]A)I_{\fH_N}
\ea
$$

On the other hand, for each $A\in\cL(\fH)$ such that $[\Dlt,A]\in\cL(\fH)$, one has
$$
\ba
(\ad^*(-\tfrac12\hbar^2\Dlt)\scrR(t))A=&-\scrR(t)[-\tfrac12\hbar^2\Dlt,A]
\\
=&-\Tr_\fH(R(t)[-\tfrac12\hbar^2\Dlt,A])I_{\fH_N}
\\
=&\Tr_\fH([-\tfrac12\hbar^2\Dlt,R(t)]A)I_{\fH_N}\,.
\ea
$$

In other words
$$
\ba
\left(i\hbar\d_t\scrR(t)-\ad^*(-\tfrac12\hbar^2\Dlt)\scrR(t)+\cC[V,\scrR(t),\scrR(t)]\right)A
\\
=\Tr_\fH\left((i\hbar\d_tR(t)-[-\tfrac12\hbar^2\Dlt+V_{R(t)},R(t)])A\right)I_{\fH_N}
\ea
$$
for each $A\in\cL(\fH)$ such that $[\Dlt,A]\in\cL(\fH)$.

Thus, $\scrR(t)$ is a solution to the evolution equation \eqref{EMQEq} if and only if $R(t)$ is a solution to the Hartree equation \eqref{HartreeEq}.
\end{proof}


\section{Uniformity in $\hbar$ of the Mean-Field Limit:\\ Proof of Theorem \ref{T-MFUnif}}\lb{S-EMMFh}


The proof of Theorem \ref{T-MFUnif} is quite involved, and will be split in several steps.

\smallskip
\noindent
\textit{Step 1.}  Set $\scrR(t)A:=\Tr_{\fH}(R(t)A)I_{\fH_N}$ for all $A\in\cL(\fH)$. By Theorems \ref{T-EMQEq} and \ref{T-Hartree}, one has
$$
\ba
i\hbar\d_t(\cM_N(t)-\scrR(t))=&\ad^*\left(-\tfrac12\hbar^2\Dlt\right)(\cM_N(t)-\scrR(t))
\\
&+\cC[V,\scrR(t),\scrR(t)]-\cC[V,\cM_N(t),\cM_N(t)]\,.
\ea
$$
Formula \eqref{DefC} shows that $(\L_1,\L_2)\mapsto\cC[V,\L_1,\L_2]$ is $\bC$-bilinear on its domain of definition. More precisely, let $\L_1,\L_2,\L'_1,\L'_2\in\cL(\cL(\fH),\cL(\fH_N)$ satisfy
$$
(\L_1,\L_2(\bu A))\,,\quad(\L'_1,\L'_2(\bu A))\,,\quad(\L_2(A\bu),\L_1)\hbox{ and }(\L'_2(A\bu),\L'_1)\in\cE_N
$$
for all $A\in\cL(\fH)$. Then, for all $\l_1,\l_2,\l'_1,\l'_2\in\bC$, one has
$$
((\l_1\L_1+\l'_1\L'_1),(\l_2\L_2+\l'_2\L'_2)(\bu A))\hbox{ and }((\l_2\L_2+\l'_2\L'_2)(A\bu ),(\l_1\L_1+\l'_1\L'_1))\in\cE_N
$$
and
$$
\ba
\cC[V,\l_1\L_1+\l'_1\L'_1,\l_2\L_2+\l'_2\L'_2]=&\l_1\l_2\cC[V,\L_1,\L_2]+\l'_1\l_2\cC[V,\L'_1,\L_2]
\\
&+\l_1\l'_2\cC[V,\L_1,\L'_2]+\l'_1\l'_2\cC[V,\L'_1,\L'_2]\,.
\ea
$$
Hence
$$
\ba
\cC[V,\scrR(t),\scrR(t)]-\cC[V,\cM_N(t),\cM_N(t)]=&\,\,\cC[V,\scrR(t),\scrR(t)-\cM_N(t)]
\\
&\,\,-\cC[V,\cM_N(t)-\scrR(t),\cM_N(t)]\,.
\ea
$$
On the other hand, Lemma \ref{L-EMQMargi1} shows that $\cM_N(t)$ is the adjoint of the continuous linear map
$$
\cL^1_s(\fH_N)\ni G\mapsto G_{:1}\in\cL^1(\fH)\,,
$$
where $G_{:1}$ is the unique element of $\cL^1(\fH)$ such that
$$
\Tr_\fH(GA)=\Tr_{\fH_N}(G(J_1A))\,,\quad\hbox{ for all }A\in\cL(\fH)\,.
$$
Therefore, one has 
$$
\cC[V,\scrR(t),\scrR(t)-\cM_N(t)]=\ad^*(V_{R(t)})(\cM_N(t)-\scrR(t))
$$
by Lemma \ref{L-LRcP}, and hence
\be\lb{EqM-R}
\ba
i\hbar\d_t(\cM_N(t)-\scrR(t))=&\ad^*\left(-\tfrac12\hbar^2\Dlt+V_{R(t)}\right)(\cM_N(t)-\scrR(t))
\\
&-\cC[V,\cM_N(t)-\scrR(t),\cM_N(t)]\,.
\ea
\ee

\smallskip
\noindent
\textit{Step 2.} Let $A^{in}\in\cL(\fH)$, and let $t\mapsto A(t)$ be the solution to the Cauchy problem
$$
\left\{
\ba
{}&i\hbar\d_tA(t)=[-\tfrac12\hbar^2\Dlt+V_{R(t)},A(t)]\,,
\\
&A(0)=A^{in}\,.
\ea
\right.
$$
Thus
$$
\ba
i\hbar\d_t((\mu^*_N(t)-\scrR(t))(A(t)))=(\cM_N(t)-\scrR(t))(i\hbar\d_tA(t))
\\
+\left(\ad^*\left(-\tfrac12\hbar^2\Dlt+V_{R(t)}\right)(\cM_N(t)-\scrR(t))\right)A(t)
\\
-\cC[V,\cM_N(t)-\scrR(t),\cM_N(t)]A(t)
\\
=(\cM_N(t)-\scrR(t))(i\hbar\d_tA(t)-[-\tfrac12\hbar^2\Dlt+V_{R(t)},A(t)])
\\
-\cC[V,\cM_N(t)-\scrR(t),\cM_N(t)]A(t)
\\
=-\cC[V,\cM_N(t)-\scrR(t),\cM_N(t)]A(t)&\,.
\ea
$$
Hence
$$
\ba
(\cM_N(t)-\scrR(t))(A(t))=&\,\,(\cM_N(0)-\scrR(0))(A(0))
\\
&\,\,-\frac{i}{\hbar}\int_0^t\cC[V,\cM_N(s)-\scrR(s),\cM_N(s)]A(s)\dd s\,.
\ea
$$

\smallskip
\noindent
\textit{Step 3.} Let $W(t)$ be a unitary operator on $\fH_N$ such that
$$
\ba
\|(\cM_N(t)-\scrR(t))(A(t))F^{in}_N\|_{\cL^1(\fH_N)}
\\
=\Tr_{\fH_N}((\cM_N(t)-\scrR(t))(A(t))F^{in}_NW(t))&\,;
\ea
$$
then
\be\lb{IntIneq0}
\ba
\|(\cM_N(t)\!-\!\scrR(t))(A(t))F^{in}_N\|_{\cL^1(\fH_N)}
\\
=\Tr_{\fH_N}((\cM_N(0)\!-\!\scrR(0))(A(0))F^{in}_NW(t))
\\
-\frac{i}{\hbar}\int_0^t\Tr_{\fH_N}((\cC[V,\cM_N(s)-\scrR(s),\cM_N(s)]A(s))F^{in}_NW(t))\dd s
\\
\le\|(\cM_N(0)-\scrR(0))(A(0))F^{in}_N\|_{\cL^1(\fH_N)}
\\
+\frac1{\hbar}\int_0^t|\Tr_{\fH_N}((\cC[V,\cM_N(s)-\scrR(s),\cM_N(s)]A(s))F^{in}_NW(t))|\dd s&\,.
\ea
\ee
Returning to the formula \eqref{DefC}, we see that
\be\lb{|C|<}
\ba
|\Tr_{\fH_N}((\cC[V,\cM_N(s)-\scrR(s),\cM_N(s)]A(s))F^{in}_NW(t))|&
\\
=
\tfrac1{(2\pi)^d}\left|\int_{\bR^d}\Tr_{\fH_N}((\cC[E_\om,\cM_N(s)-\scrR(s),\cM_N(s)]A(s))F^{in}_NW(t))\hat V(\dd\om)\right|
\\
\le\tfrac1{(2\pi)^d}\int_{\bR^d}|\Tr_{\fH_N}((\cC[E_\om,\cM_N(s)-\scrR(s),\cM_N(s)]A(s))F^{in}_NW(t))||\hat V|(\dd\om)&\,.
\ea
\ee
Since
$$
\ba
\cC[E_\om,\cM_N(s)\!-\!\scrR(s),\cM_N(s)]A(s)=&\,((\cM_N(s)\!-\!\scrR(s))E^*_\om)(\cM_N(s)(E_\om A(s)))
\\
&\,-\!(\cM_N(s)(A(s)E_\om))((\cM_N(s)\!-\!\scrR(s))E^*_\om)
\\
=&\,(\cM_N(s)[E_\om,A(s)])((\cM_N(s)-\scrR(s))E^*_\om)
\\
&\,+[\cM_N(s)E^*_\om,\cM_N(s)(E_\om A(s))]
\\
=&\,(\cM_N(s)[E_\om,A(s)])((\cM_N(s)-\scrR(s))E^*_\om)
\\
&\,+\frac1N\cM_N(s)[E^*_\om,E_\om A(s)]\,,
\ea
$$
(where the last equality follows from Lemma \ref{L-ComMu}), one has
\be\lb{|Com|<0}
\ba
|\Tr_{\fH_N}((\cC[E_\om,\cM_N(s)-\scrR(s),\cM_N(s)]A(s))F^{in}_NW(t))|
\\
\le\|W(t)(\cM_N(s)[E_\om,A(s)])\|\|((\cM_N(s)-\scrR(s))E^*_\om)F^{in}_N\|_{\cL^1(\fH_N)}
\\
+\frac1N\|W(t)(\cM_N(s)[E^*_\om,E_\om A(s)])\|\|F^{in}_N\|_{\cL^1(\fH_N)}&\,.
\ea
\ee

By construction
$$
\ba
\|\cM_N(t)A\|=\|U_N(t)(\cM^{in}_NA)U_N(t)^*\|
\\
=\|\cM^{in}_NA\|\le\frac1N\sum_{k=1}^N\|J^*_kA\|=\|A\|&\,.
\ea
$$
Since $W(t)$ is unitary, the inequality above implies that
$$
\|W(t)(\cM_N(s)[E_\om,A(s)])\|\le\|[E_\om,A(s)]\|\,,
$$
and
$$
\|W(t)(\cM_N(s)[E^*_\om,E_\om A(s)])\|\le\|[E^*_\om,E_\om A(s)]\|=\|[E^*_\om,A(s)]\|
$$
because $[E^*_\om,E_\om A(s)]=E_\om[E^*_\om,A(s)]$ and $E_\om$ is unitary. 

\smallskip
Since $F^{in}_N\in\cD_s(\fH_N)$, we conclude from the two previous inequalities and \eqref{|Com|<0} that
\be\lb{|Com|<}
\ba
|\Tr_{\fH_N}((\cC[E_\om,\cM_N(s)-\scrR(s),\cM_N(s)]A(s))F^{in}_NW(t))|
\\
\le\|[E_\om,A(s)]\|\|((\cM_N(s)-\scrR(s))E^*_\om)F^{in}_N\|_{\cL^1(\fH_N)}+\frac1N\|[E^*_\om,A(s)]\|&\,.
\ea
\ee

\smallskip
\noindent
\textit{Step 4.} Pick $T>0$; for each $t\in\bR$, define
\be\lb{DefdNT}
d^T_N(t):=\sup_{B\in\cW[T]}\|((\cM_N(t)-\scrR(t))(S(t,0)BS(0,t)))F^{in}_N\|_{\cL^1(\fH_N)}\,,
\ee
where
$$
\cW[T]:=\bigcup_{0\le\tau\le T}\cV[\tau]\,,\quad\hbox{ with }\cV[\tau]:=\left\{S(0,\tau)\OP^W_{\hbar}[a]S(\tau,0)\,:\,a\in\cB^{d+5}\right\}\,.
$$
Since $E_\om$ is the Weyl operator with symbol $(x,\xi)\mapsto e^{i\om\cdot x}$, one has obviously
$$
\max(1,|\om|)^{-d-5}E_\om\in\cV[0]=\cW[0]\,.
$$
Hence
$$
\ba
\|((\cM_N(s)-\scrR(s))E_\om)F^{in}_N\|_{\cL^1(\fH_N)}
\\
\le\max(1,|\om|)^{d+5}\sup_{a\in\cB^{d+5}}\|((\cM_N(s)-\scrR(s))\OP^W_{\hbar}[a])F^{in}_N\|_{\cL^1(\fH_N)}
\\
=\max(1,|\om|)^{d+5}\sup_{B\in\cV[s]}\|((\cM_N(s)-\scrR(s))(S(s,0)BS(0,s)))F^{in}_N\|_{\cL^1(\fH_N)}
\\
\le \max(1,|\om|)^{d+5}d^T_N(s)&\,.
\ea
$$
Therefore, inequality \eqref{|Com|<} becomes
\be\lb{|Com|<ter}
\ba
\frac1\hbar|\Tr_{\fH_N}((\cC[E_\om,\cM_N(s)-\scrR(s),\cM_N(s)]A(s))F^{in}_NW(t))|
\\
\le\frac1\hbar\|[E_\om,A(s)]\|\max(1,|\om|)^{d+5}d^T_N(s)+\frac1{N\hbar}\|[E^*_\om,A(s)]\|&\,.
\ea
\ee

We are left with the task of estimating 
$$
\frac{\|[E_\om,A(s)]\|}\hbar\quad\hbox{ for all }A(s)=S(s,0)A^{in}S(0,s)\hbox{ with }A^{in}\in\cW[T]\,.
$$
In other words, $A(s)$ is of the form
$$
S(s,\tau)\OP^W_{\hbar}[a]S(\tau,s)\quad\hbox{ with }a\in\cB^{d+5}\,.
$$
The key estimate for all such operators is provided by the following lemma.

\begin{Lem}\lb{L-KeyL}
Under the assumptions of Theorem \ref{T-MFUnif}, there exists a positive constant $C_d\equiv C_d[\bV]$ such that, for all $A^{in}\in\cW[T]$, all $\om\in\bR^d$ and all $s\in[0,T]$, the operator $A(s)=S(s,0)A^{in}S(0,s)$ satisfies
$$
\frac{\|[E_\om,A(s)]\|}\hbar\le C_d[\bV]\left(|\om|e^{(d+5)\tilde\Ga_2T}+\hbar^2\frac{e^{(d+5)\tilde\Ga_2T}-1}{(d+5)\tilde\Ga_2}\bV\right)\,,
$$
where
$$
\tilde\Ga_2:=\|\grad^2V\|_{L^\infty(\bR^d)}\,.
$$
\end{Lem}

\smallskip
The proof of Lemma \ref{L-KeyL} rests on semiclassical estimates for Weyl operators propagated by a time-dependent Schr\"odinger operator --- specifically, by the adjoint Hartree flow --- and is rather lengthy. This estimate is 
based on mathematical techniques  significantly different from the mathematical apparatus used in most of this paper. We shall therefore take Lemma \ref{L-KeyL} for granted, finish the proof of Theorem \ref{T-MFUnif}, and 
postpone the proof of Lemma \ref{L-KeyL} until the next section.

\smallskip
Inserting the bound provided by Lemma \ref{L-KeyL} shows that for all $s,t\in[0,T]$ and all $A^{in}\in\cW[T]$, one has
\be\lb{|Com|<4to}
\ba
\frac1\hbar|\Tr_{\fH_N}((\cC[E_\om,\cM_N(s)-\scrR(s),\cM_N(s)]A(s))F^{in}_NW(t))|
\\
\le C_d[\bV]e^{(d+5)\tilde\Ga_2T}\left(|\om|+\hbar^2\frac{\bV}{(d+5)\tilde\Ga_2}\right)\left(\max(1,|\om|)^{d+5}d^T_N(s)+\frac1N\right)&\,.
\ea
\ee

\smallskip
\noindent
\textit{Step 5.} Now we use the inequality \eqref{|Com|<4to} to bound the right hand side of the integral inequality \eqref{IntIneq0}. One finds that, for each $t\in[0,T]$ and each $A^{in}\in\cW[T]$,
\be\lb{IntIneq3}
\ba
\|(\cM_N(t)\!-\!\scrR(t))(A(t))F^{in}_N\|_{\cL^1(\fH_N)}\!\le\!\|((\cM_N(0)-\scrR(0))A^{in})F^{in}_N\|_{\cL^1(\fH_N)}
\\
+C_d[\bV]e^{(d+5)\tilde\Ga_2T}\!\!\!\int_0^t\!\!\int_{\bR^d}\left(|\om|\!+\!\hbar^2\tfrac{\bV}{(d+5)\tilde\Ga_2}\right)\left(\max(1,|\om|)^{d+5}d^T_N(s)\!+\!\tfrac1N\right)|\hat V(\om)|\tfrac{\dd\om\dd s}{(2\pi)^d}
\\
\le d^T_N(0)\!+\!C_d[\bV]\bV\left(1\!+\!\hbar^2\tfrac{\bV}{(d+5)\tilde\Ga_2}\right)e^{(d+5)\tilde\Ga_2T}\!\!\!\int_0^t\left(d^T_N(s)\!+\!\frac1N\right)\dd s&\,,
\ea
\ee
since
$$
\|((\cM_N(0)-\scrR(0))A^{in})F^{in}_N\|_{\cL^1(\fH_N)}\le d^T_N(0)\,.
$$
Observing that $A(t)=S(t,0)A^{in}S(0,t)$ with $t\in[0,T]$, and maximizing both sides of the inequality \eqref{IntIneq3} as $A^{in}$ runs through $\cW[T]$, one finds that
\be\lb{IntIneq4}
d^T_N(t)\le d^T_N(0)+C_d[\bV]\bV\left(1+\hbar^2\frac{\bV}{(d+5)\tilde\Ga_2}\right)e^{(d+5)\tilde\Ga_2T}\int_0^t\left(d^T_N(s)+\frac1N\right)\dd s
\ee
for all $t\in[0,T]$. Applying Gronwall's lemma to the integral inequality \eqref{IntIneq4} shows that
\be\lb{GronwdNT}
d^T_N(T)+\frac1N\le \left(d^T_N(0)+\frac1N\right)\exp\left(C_d[\bV] Te^{(d+5)\tilde\Ga_2T}\bV\left(1+\hbar^2\frac{\bV}{(d+5)\tilde\Ga_2}\right)\right)\,.
\ee

\smallskip
\noindent
\textit{Step 6.} Next, observe that
$$
\ba
d^T_N(T)\ge&\sup_{B\in\cV[T]}\|((\cM_N(T)-\scrR(T))(S(T,0)BS(0,T)))F^{in}_N\|_{\cL^1(\fH_N)}
\\
=&\sup_{B\in\cV[0]}|\Tr_{\fH_N}(((\cM_N(T)-\scrR(T))B)F^{in}_N)|
\\
=&\sup_{B\in\cV[0]}|\Tr_\fH((F_{N:1}(T)-R(T))B)|\,.
\ea
$$
On the other hand
$$
\ba
\Tr_\fH((F_{N:1}(T)-R(T))\OP^W_{\hbar}[a])
\\
=\iint_{\bR^d\times\bR^d}\left(W_{\hbar}[F_{N:1}(T)]-W_{\hbar}[R(T)]\right)(x,\xi)a(x,\xi)\dd x\dd\xi&\,,
\ea
$$
according to formula \eqref{DualWOpW} in Appendix B. Hence
\be\lb{dNT(T)>}
\ba
d^T_N(T)\ge\sup_{a\in\cB^{d+5}}\left|\iint_{\bR^d\times\bR^d}\left(W_{\hbar}[F_{N:1}(T)]-W_{\hbar}[R(T)]\right)(x,\xi)a(x,\xi)\dd x\dd\xi\right|
\\
=\left\|W_{\hbar}[F_{N:1}(T)]-W_{\hbar}[R(T)]\right\|'_{d+5,\infty}
\ea
\ee

\smallskip
\noindent
\textit{Step 7.} Next we seek an upper bound for $d^T_N(0)$, to be inserted on the right hand side of \eqref{GronwdNT}. By the Calderon-Vaillancourt theorem (see Theorem \ref{T-CV} in Appendix B), one has
$$
\cV[0]\subset \overline{B(0,\g_d)}_{\cL(\fH)}\,.
$$
Hence
\be\lb{dNT(0)<}
\ba
d^T_N(0):=&\sup_{B\in\cV[0]}\|((\cM_N(0)-\scrR(t))B)F^{in}_N\|_{\cL^1(\fH_N)}
\\
\le&\sup_{\|B\|\le\g_d}\|((\cM_N(0)-\scrR(t))B)F^{in}_N\|_{\cL^1(\fH_N)}
\\
\le&\g_d\sup_{\|B\|\le 1}\left\|((\cM_N(0)-\scrR(0))B)\sqrt{F^{in}_N}\sqrt{F_N^{in}}\right|
\\
\le&\g_d\sup_{\|B\|\le 1}\left\|((\cM_N(0)-\scrR(0))B)\sqrt{F^{in}_N}\right\|_{\cL^2(\fH_N)}\,,
\ea
\ee
so that one is left with the task of computing
$$
\left\|((\cM_N(0)-\scrR(0))B)\sqrt{F^{in}_N}\right\|_{\cL^2(\fH_N)}\,.
$$
One expands
$$
\ba
\left\|((\cM_N(0)-\scrR(0))B)\sqrt{F^{in}_N}\right\|^2_{\cL^2(\fH_N)}
\\
=\Tr_{\fH_N}\left(\sqrt{F^{in}_N}((\cM_N(0)-\scrR(0))B)^*((\cM_N(0)-\scrR(0))B)\sqrt{F^{in}_N}\right)
\\
=\Tr_{\fH_N}(((\cM_N(0)-\scrR(0))B)^*((\cM_N(0)-\scrR(0))B)F^{in}_N)
\\
=\Tr_{\fH_N}((\cM_N(0)B)^*(\cM_N(0)B)F^{in}_N)-\Tr_{\fH_N}((\cM_N(0)B)^*(\scrR(0)B)F^{in}_N)
\\
-\Tr_{\fH_N}((\scrR(0)B)^*(\cM_N(0)B)F^{in}_N)+\Tr_{\fH_N}((\scrR(0)B)^*(\scrR(0)B)F^{in}_N)&\,.
\ea
$$
One has
$$
\ba
\Tr_{\fH_N}((\scrR(0)B)^*(\scrR(0)B)F^{in}_N)=&|\Tr_\fH(R(0)B)|^2\Tr_{\fH_N}(F^{in}_N)
\\
=&|\Tr_\fH(R(0)B)|^2\,,
\ea
$$
and
$$
\ba
\Tr_{\fH_N}((\scrR(0)B)^*(\cM_N(0)B)F^{in}_N)=&\overline{\Tr_\fH((R(0)B)}\Tr_{\fH_N}((\cM_N(0)B)F^{in}_N)
\\
=&\overline{\Tr_\fH(R(0)B)}\Tr_{\fH}(BF^{in}_{N:1})\,.
\ea
$$
Then
$$
\ba
\Tr_{\fH_N}((\mu^*_N(0)(B))^*\scrR(0)(B)F^{in}_N)=\Tr_{\fH_N}(F^{in}_N(\mu^*_N(0)(B))^*\scrR(0)(B))
\\
=\overline{\Tr_{\fH_N}((F^{in}_N(\mu^*_N(0)(B))^*\scrR(0)(B))^*)}=\overline{\Tr_{\fH_N}(\scrR(0)(B)^*\mu^*_N(0)(B)F^{in}_N)}
\\
=\Tr_\fH(R(0)B)\overline{\Tr_{\fH}(BF^{in}_{N:1})}&\,.
\ea
$$
It remains to compute
$$
\ba
\Tr_{\fH_N}((\cM_N(0)B)^*(\cM_N(0)B)F^{in}_N)
\\
=\frac1{N^2}\sum_{1\le k\not=l\le N}\Tr_{\fH_N}\Tr_{\fH_N}((J_kB^*)(J_lB)F^{in}_N)
\\
+\frac1{N^2}\sum_{n=1}^N\Tr_{\fH_N}((J_kB^*)(J_kB)F^{in}_N)
\\
=\frac{N-1}N\Tr_{\fH_2}((B^*\otimes B)F^{in}_{N:2})+\frac1N\Tr_{\fH}(B^*BF^{in}_{N:1})&\,.
\ea
$$

Summarizing, we have found that
$$
\ba
\left\|(\cM_N(0)-\scrR(0))(B)\sqrt{F^{in}_N}\right\|^2_{\cL^2(\fH_N)}
\\
=\frac{N-1}N\Tr_{\fH_2}((B^*\otimes B)F^{in}_{N:2})+\frac1N\Tr_{\fH}(B^*BF^{in}_{N:1})
\\
-2\Re\left(\overline{\Tr_\fH(R(0)B)}\Tr_{\fH}(BF^{in}_{N:1})\right)+|\Tr_\fH(R(0)B)|^2&\,.
\ea
$$
One can rearrange this term as
$$
\ba
\left\|((\cM_N(0)-\scrR(0))B)\sqrt{F^{in}_N}\right\|^2_{\cL^2(\fH_N)}
\\
=\frac{N-1}N\Tr_{\fH_2}((B^*\otimes B)(F^{in}_{N:2}-F^{in}_{N:1}\otimes F^{in}_{N:1}))
\\
+\frac1N(\Tr_{\fH}(B^*BF^{in}_{N:1})-|\Tr_\fH(BF^{in}_{N:1})|^2)
\\
+|\Tr_{\fH}(B(F^{in}_{N:1}-R(0)))|^2&\,.
\ea
$$
Hence
\be\lb{dN(0)<}
\ba
\sup_{\|B\|\le 1}\left\|((\cM_N(0)-\scrR(0))B)\sqrt{F^{in}_N}\right\|^2_{\cL^2(\fH_N)}
\\
\le\frac{N-1}N\|F^{in}_{N:2}-F^{in}_{N:1}\otimes F^{in}_{N:1}\|_{\cL^1(\fH_2)}+\|F^{in}_{N:1}-R(0)\|_{\cL^1(\fH)}^2+\frac1N&\,.
\ea
\ee

Since we have assumed in Theorem \ref{T-MFUnif} that $F^{in}_N=(R^{in})^{\otimes N}$, one has
$$
F^{in}_{N:1}=R^{in}=R(0)\quad\hbox{ and }\quad F^{in}_{N:2}=(R^{in})^{\otimes 2}=F^{in}_{N:1}\otimes F^{in}_{N:1}\,.
$$
Therefore, we conclude from \eqref{dNT(0)<}  and  \eqref{dN(0)<} that
\be\lb{dNT(0)<fin}
d^T_N(0)\le\frac{\g_d}{\sqrt{N}}\,.
\ee

Inserting the bounds \eqref{dNT(0)<fin} and \eqref{dNT(T)>} in \eqref{GronwdNT}, and setting $\Ga=\tilde\Ga_2$, we conclude that
$$
\ba
\left\|W_{\hbar}[F_{N:1}(T)]-W_{\hbar}[R(T)]\right\|'_{d+5,\infty}
\\
\le\left(\g_dd_N(0)+\frac1N\right)\exp\left(C_d[\bV] Te^{(d+5)\Ga T}\bV\left(1+\hbar^2\frac{\bV}{(d+5)\Ga}\right)\right)
\\
\le\frac{\g_d+1}{\sqrt{N}}\exp\left(C_d[\bV] Te^{(d+5)\Ga T}\bV\left(1+\hbar^2\frac{\bV}{(d+5)\Ga}\right)\right)&\,,
\ea
$$
which is precisely the desired inequality.

\smallskip
\noindent
\textbf{Remark.} By comparison with the proofs of Theorem 2.4 in \cite{FGMouPaul}, a striking feature of the present proof is that it uses a different distance for each time $t$ at which we seek to compare the Hartree solution $R(t)$ and the first 
marginal $F_{N:1}(t)$ of the $N$-particle density. Specifically, for each $T>0$, the distance $d^T_N(T)$ is used to estimate the difference, $W_{\hbar}[F_{N:1}(T)]-W_{\hbar}[R(T)]$. On the contrary, in \cite{FGMouPaul}, the convergence rate for 
all time intervals is estimated in terms of a single pseudo-distance constructed by analogy with the quadratic Monge-Kantorovich (or Wasserstein) distance used in optimal transport.


\section{Proof of Lemma \ref{L-KeyL}}\lb{S-SC}


Let $\scrV\equiv\scrV(t,x)$ be a continuous, real-valued, time-dependent potential defined on $\bR\times\bR^d$ such that, for each $t\in\bR$, the unbounded operator
$$
-\tfrac12\hbar^2\Dlt_x+\scrV(t,x)
$$
has a self-adjoint extension on $L^2(\bR^d)$. For each $\tau\in\bR$, denote by $t\mapsto U_{\hbar}(t,\tau)$ the time-dependent unitary operator on $\fH$ such that
$$
\left\{
\ba
{}&i\hbar\d_tU_{\hbar}(t,\tau)=(-\tfrac12\hbar^2\Dlt_x+\scrV(t,x))U_{\hbar}(t,\tau)\,,
\\
&U_{\hbar}(\tau,\tau)=I_\fH\,,
\ea
\right.
$$
and set
$$
\scrB_{\hbar}(t,s):=U_{\hbar}(t,s)\OP^W_{\hbar}[b]U_{\hbar}(s,t)\,.
$$
Henceforth we denote by $\hat\scrV$ the Fourier transform of $\scrV$ in the $x$-variable.

\begin{Lem}\lb{L-[E,B]}
Let $\scrV\equiv\scrV(t,x)\in C_b(\bR\times\bR^d)$ be a real-valued potential satisfying
$$
\bW:=\tfrac1{(2\pi)^d}\int_{\bR^d}\sup_{\tau\in\bR}|\hat\scrV(\tau,\xi)|(1+|\xi|)^{d+6}\dd\xi<\infty\,.
$$
Then there exists $C_d\equiv C_d[\bW]>0$ that is a nondecreasing function of $\bW$ such that, for each $b\equiv b(x,\xi)\in C^{d+5}_b(\bR^d\times\bR^d)$, each $s,t\in\bR$, each $\om\in\bR^d$ and each $\hbar\in(0,1]$, one has
$$
\ba
\left\|\frac{[E_\om,\scrB_{\hbar}(t,s)]}\hbar\right\|\le C_d\|b\|_{d+5,\infty}\left(|\om|e^{(d+5)\Ga_2|t-s|}+\hbar^2\,\frac{e^{(d+5)\Ga_2|t-s|}-1}{(d+5)\Ga_2}\bW\right)&\,,
\ea
$$
where 
$$
\Ga_2:=\|\grad^2_x\scrV\|_{L^\infty(\bR\times\bR^d)}\,.
$$
\end{Lem}

\smallskip
On the other hand, let $t\mapsto\Phi(t,\tau;x,\xi)$ be the solution of the Cauchy problem for the Hamiltonian system
$$
\left\{
\ba
{}&\dot{X}(t)=\Xi(t)\,,&&\qquad X(\tau)=x\,,
\\
&\dot\Xi(t)=-\grad_X\scrV(t,X(t))\,,&&\qquad\Xi(\tau)\,=\xi\,.
\ea
\right.
$$

We seek to compare $\scrB(t,s)$ with
$$
B_{\hbar}(t,s):=\OP^W_{\hbar}[b\circ\Phi(t,s)]\,,
$$
where $\Phi(t,s)$ designates the diffeomorphism $\bR^{2d}\ni(x,\xi)\mapsto\Phi(t,s,x,\xi)\in\bR^{2d}$. The proof of Lemma \ref{L-[E,B]} rests on Lemmas \ref{L-EstimB} and \ref{L-scrB-B} stated below.

\subsection{Estimating $[E_\om,B_{\hbar}(t,s)]$}

\begin{Lem}\lb{L-BR}
Assume that $\scrV\equiv\scrV(t,x)\in C_b(\bR\times\bR^d)$ is $d+6$ times differentiable in $x$ and satisfies $\grad^m_x\scrV\in C_b(\bR\times\bR^d)$ for all $m=1,\ldots,d+6$. Set
$$
\sup_{(t,x)\in\bR^{d+1}}|\grad^m_x\scrV(t,x)|=:\Ga_m<\infty\,,\quad\hbox{ for }m=1,\ldots,d+6\,.
$$
(a) For each $\nu\in\bN^d$ such that $|\nu|\le d+5$, there exists a positive constant 
$$
G_\nu\equiv G_\nu[\max(\Ga_1,\ldots,\Ga_{|\nu|+1})]
$$
that is a nondecreasing function of $\max(\Ga_1,\ldots,\Ga_{|\nu|+1})$ such that, for each $s,t\in\bR$,
$$
\|\d^\nu\Phi(t,s)\|_{L^\infty(\bR^d\times\bR^d)}\le G_\nu e^{\Ga_2|\nu||t-s|}\,.
$$
(b) For each $\nu\in\bN^d$ such that $|\nu|\le d+5$, there exists a positive constant 
$$
F_\nu\equiv F_\nu[\max(\Ga_1,\ldots,\Ga_{|\nu|+1})]
$$ 
that is a nondecreasing function of $\max(\Ga_1,\ldots,\Ga_{|\nu|+1})$ such that, for each $s,t\in\bR$,
$$
\|\d^\nu(b\circ\Phi(t,s))\|_{L^\infty(\bR^d\times\bR^d)}\le F_\nu e^{\Ga_2|\nu||t-s|}\max_{\mu\le\nu}\|\d^\mu b\|_{L^\infty(\bR^d\times\bR^d)}\,.
$$
\end{Lem}

Statement (a) is a variant of Lemma 2.2 of \cite{BR} with time-dependent potential $V$; adapting the proof of \cite{BR} to the present case is obvious. Statement (b) follows from (a) by the Fa\`a di Bruno formula: see Lemma 2.4 in \cite{BR},
for an analogous statement, where the dependence of the upper bound in terms of the derivatives of $b$ is not specified.

\begin{Lem}\lb{L-EstimB}
Set 
$$
\cF_{d+3}:=\max_{|\nu|\le d+3}F_\nu[\max(\Ga_1,\ldots,\Ga_{|\nu|+1})]
$$
where $F_\nu$ is the positive constant which appears in Lemma \ref{L-BR} (b). Then, for each $s,t\in\bR$, the operator
$$
B_{\hbar}(t,s):=\OP^W_{\hbar}[b\circ\Phi(t,s)]
$$
satisfies
$$
\left\|\frac1{\hbar}[E_\om,B_{\hbar}(t,s)]\right\|\le\g_d\cF_{d+3}e^{(d+3)\Ga_2|t-s|}|\om|\|b\|_{d+3,\infty}\,,
$$
where $\g_d>0$ is the constant which appears in the Calderon-Vaillancourt theorem (Theorem \ref{T-CV} below).
\end{Lem}

\begin{proof}
First, observe that
$$
[E_\om,\OP^W_{\hbar}[b]]=\left(E_\om\OP^W_{\hbar}[b]E_\om^*-\OP^W_{\hbar}[b]\right)E_\om\,,
$$
and that the operator $E_\om\OP^W_{\hbar}[b]E_\om^*$ has integral kernel
$$
\int_{\bR^d}b\left(\frac{x+y}2,\hbar\xi\right)e^{i(\xi+\om)\cdot(x-y)}\frac{\dd\xi}{(2\pi)^d}\,,
$$
so that
$$
E_\om\OP^W_{\hbar}[b]E_\om^*=\OP^W_{\hbar}[b(\cdot,\cdot-\hbar\om)]\,.
$$
Hence
$$
\|[E_\om,\OP^W_{\hbar}[b]]\|=\|\OP^W_{\hbar}[b(\cdot,\cdot-\hbar\om)-b]\|\,,
$$
and by the Calderon-Vaillancourt theorem (Theorem \ref{T-CV} below)
$$
\left\|\frac1{\hbar}[E_\om,\OP^W_{\hbar}[b]]\right\|\le\g_d|\om|\sup_{|\a|\le[d/2]+1\atop|\b|\le[d/2]+2}\|\d_x^\a\d_\xi^\b b\|_{L^\infty}\,.
$$
Since $B_{\hbar}(t,s):=\OP^W_{\hbar}[b\circ\Phi(t,s)]$, this inequality becomes
$$
\left\|\frac1{\hbar}[E_\om,B_{\hbar}(t,s)]\right\|\le\g_d|\om|\sup_{|\a|\le[d/2]+1\atop|\b|\le[d/2]+2}\|\d_x^\a\d_\xi^\b b\circ\Phi(t,s)\|_{L^\infty}\,,
$$
and one concludes with Lemma \ref{L-BR} (b).
\end{proof}

\subsection{Estimating $\scrB_{\hbar}(t,s)-B_{\hbar}(t,s)$}

\begin{Lem}\lb{L-scrB-B}
Set 
$$
\cF_{d+5}:=\max_{|\nu|\le d+5}F_\nu[\max(\Ga_1,\ldots,\Ga_{|\nu|+1})]
$$
where $F_\nu$ is the positive constant which appears in Lemma \ref{L-BR} (b). Then, for each $s,t\in\bR$, the operators
$$
B_{\hbar}(t,s):=\OP^W_{\hbar}[b\circ\Phi(t,s)]\quad\hbox{ and }\quad\scrB_{\hbar}(t,s):=U_{\hbar}(t,s)\OP^W_{\hbar}[b]U_{\hbar}(s,t)
$$
satisfy the bound
$$
\ba
\|\scrB_{\hbar}(t,s)-B(t,s)\|\le\tfrac1{24}\g_d\cF_{d+5}\frac{e^{\Ga_2(d+5)|t-s|}-1}{\Ga_2(d+5)}\hbar^2\sup_{|\a|\le[d/2]+1\atop |\b|\le[d/2]+4}\|\d_q^{\a}\d_p^\b b\|_{L^\infty} 
\\
\times\int_{\bR^d}\sup_{\tau\in\bR}|\hat\scrV(\tau,\xi)|(1+|\xi|)^{[d/2]+4}\tfrac{\dd\xi}{(2\pi)^d}&\,.
\ea
$$
where $\g_d>0$ is the constant which appears in the Calderon-Vaillancourt theorem (Theorem \ref{T-CV} below).
\end{Lem}

\begin{proof}
One has
$$
i\hbar\d_t\scrB_{\hbar}(t,s)=[-\tfrac12\hbar^2\Dlt_x+\scrV(t,x),\scrB_{\hbar}(t,s)]\,,\qquad \scrB_{\hbar}(s,s)=\OP^W_{\hbar}[b]\,,
$$
and
$$
i\hbar\d_tB_{\hbar}(t,s)=\OP^W_{\hbar}(-i\hbar\{\tfrac12|\xi|^2+\scrV(t,x),b\circ\Phi(t,s)\})\,,\qquad B_{\hbar}(s,s)=\OP^W_{\hbar}[b]\,.
$$
Thus
$$
i\hbar\d_t(\scrB_{\hbar}(t,s)-B_{\hbar}(t,s))=[-\tfrac12\hbar^2\Dlt_x+\scrV(t,x),\scrB_{\hbar}(t,s)-B_{\hbar}(t,s)]+Q_{\hbar}(t,s)\,,
$$
where
$$
Q_{\hbar}(t,s)=[-\tfrac12\hbar^2\Dlt_x+\scrV(t,x),B_{\hbar}(t,s)]+\OP^W_{\hbar}(i\hbar\{\tfrac12|\xi|^2+\scrV(t,x),b\circ\Phi(t,s)\})\,.
$$
Hence
$$
\scrB_{\hbar}(t,s)-B_{\hbar}(t,s)=\frac1{i\hbar}\int_s^tU_{\hbar}(t,\tau)Q_{\hbar}(\tau,s)U_{\hbar}(\tau,t)\dd\tau\,.
$$
Since $U_{\hbar}(t,s)$ is unitary for all $s,t\in\bR$, one has
$$
\|\scrB_{\hbar}(t,s)-B(t,s)\|\le\frac1{\hbar}\int_s^t\|Q_{\hbar}(\tau,s)\|\dd s\,,\quad\hbox{ for all }t>s\,.
$$
The norm of the source term $Q_{\hbar}$ is estimated with the following lemma.

\begin{Lem}\lb{L-Qh}
For each $a\in\cS(\bR^d\times\bR^d)$, one has
$$
\ba
\frac1{i\hbar}\left[-\tfrac12\hbar^2\Dlt_x+\scrV(t,x),\OP^W_{\hbar}[a]\right]+\OP^W_{\hbar}[\{\tfrac12|\xi|^2+\scrV(t,x),a\}]
\\
=\hbar^2\OP^W_{\hbar}[R_{\hbar}(t,\cdot)]&\,,
\ea
$$
where
$$
\ba
R_{\hbar}(t,q,p):=
\\
\tfrac1{i\pi^d}\!\!\!\int_{\bR^d}\left(\hat\scrV(t,-2\zeta)e^{-2iq\cdot\zeta}\!-\!\hat\scrV(t,2\zeta)e^{2iq\cdot\zeta}\right)\zeta^{\otimes 3}:\left(\int_0^1\tfrac{(1-\tau)^2}{2}\grad^3_pa(q,p\!+\!\tau\hbar\zeta)\dd\tau\right)\dd\zeta&\,.
\ea
$$
Moreover
$$
\ba
\left\|\OP^W_{\hbar}[R_{\hbar}(t,\cdot)]\right\|\le\tfrac{\g_d}{24}\sup_{|\a|\le[d/2]+1\atop |\b|\le[d/2]+4}\|\d_q^{\a}\d_p^\b a\|_{L^\infty} \int_{\bR^d}|\hat\scrV(t,\xi)|(1+|\xi|)^{[d/2]+4}\tfrac{\dd\xi}{(2\pi)^d}\,.
\ea
$$
\end{Lem}

\smallskip
The proof of this lemma is deferred until the next section. Applying the inequality in Lemma \ref{L-Qh} to $a=b\circ\Phi(t,s)$, and Lemma \ref{L-BR} to the symbol $b\circ\Phi(t,s)$, one finds that
$$
\ba
\|Q_{\hbar}(t,s)\|
\\
\le\tfrac{\g_d}{24}\hbar^3\sup_{|\a|\le[d/2]+1\atop |\b|\le[d/2]+4}\|\d_q^{\a}\d_p^\b(b\circ\Phi(t,s))\|_{L^\infty} \int_{\bR^d}\sup_{\tau\in\bR}|\hat\scrV(\tau,\xi)|(1+|\xi|)^{[d/2]+4}\tfrac{\dd\xi}{(2\pi)^d}
\\
\le\tfrac{\g_d}{24}cF_{d+5}e^{\Ga_2(d+5)|t-s|}\hbar^3\sup_{|\a|\le[d/2]+1\atop |\b|\le[d/2]+4}\|\d_q^{\a}\d_p^\b b\|_{L^\infty} \int_{\bR^d}\sup_{\tau\in\bR}|\hat\scrV(\tau,\xi)|(1+|\xi|)^{[d/2]+4}\tfrac{\dd\xi}{(2\pi)^d}&\,.
\ea
$$

Hence
$$
\ba
\|\scrB_{\hbar}(t,s)-B(t,s)\|\le\tfrac1{24}\g_d\cF_{d+5}\left(\int_s^te^{\Ga_2(d+5)|\tau-s|}\dd\tau\right)\hbar^2\sup_{|\a|\le[d/2]+1\atop |\b|\le[d/2]+4}\|\d_q^{\a}\d_p^\b b\|_{L^\infty}
\\
\times\int_{\bR^d}\sup_{\tau\in\bR}|\hat\scrV(\tau,\xi)|(1+|\xi|)^{[d/2]+4}\frac{\dd\zeta}{(2\pi)^d}&\,,
\ea
$$
and the estimate in Lemma \ref{L-scrB-B} follows from computing the first integral on the right hand side of this inequality.
\end{proof}

\subsection{Proof of Lemma \ref{L-Qh}}

First, the formula for the Moyal product gives the following identity:
$$
\ba
(\tfrac12|\xi|^2\#a-a\#\tfrac12|\xi|^2)(q,p)=&\frac1{\pi^d}\int_{\bR^{2d}}a(q+z,p)\tfrac12(p+\hbar\zeta)^2(e^{-2iz\cdot\zeta}-e^{2iz\cdot\zeta})\dd z\dd\zeta
\\
=&\frac{\hbar}{\pi^d}\int_{\bR^{2d}}a(q+z,p)(\zeta\cdot p)(e^{-2iz\cdot\zeta}-e^{2iz\cdot\zeta})\dd z\dd\zeta
\\
=&-i\hbar p\cdot\grad_qa(q,p)=-i\hbar\{\tfrac12|\xi|^2,a\}(q,p)\,,
\ea
$$
where the penultimate equality follows from the Fourier inversion formula. Hence
\be\lb{CommutDelta}
\frac1{i\hbar}\left[-\tfrac12\hbar^2\Dlt_x,\OP^W_{\hbar}[a]\right]+\OP^W_{\hbar}[\{\tfrac12|\xi|^2,a\}]=0\,.
\ee

Applying again the Moyal product formula shows that
$$
\ba
(\scrV(t,\cdot)\#a-a\#\scrV(t,\cdot))(q,p)
\\
=\frac1{\pi^d}\int_{\bR^{2d}}\scrV(t,q+z)a(q,p+\hbar\zeta)(e^{2iz\cdot\zeta}-e^{-2iz\cdot\zeta})\dd z\dd\zeta
\\
=\frac1{\pi^d}\int_{\bR^d}\left(\hat\scrV(t,-2\zeta)e^{-2iq\cdot\zeta}-\hat\scrV(t,2\zeta)e^{2iq\cdot\zeta}\right)a(q,p+\hbar\zeta)\dd\zeta&\,.
\ea
$$
Expand the right hand side of the identity above in powers of $\hbar$, up to order $3$; since the function 
$$
\zeta\mapsto\hat\scrV(t,-2\zeta)e^{-2iq\cdot\zeta}-\hat\scrV(t,2\zeta)e^{2iq\cdot\zeta}
$$
is odd, only the first and third order terms are kept in this expansion. Hence
\be\lb{CommutV}
\ba
(\scrV(t,\cdot)\#a-a\#\scrV(t,\cdot))(q,p)
\\
=\frac\hbar{\pi^d}\int_{\bR^d}\left(\hat\scrV(t,-2\zeta)e^{-2iq\cdot\zeta}-\hat\scrV(t,2\zeta)e^{2iq\cdot\zeta}\right)\zeta\cdot\grad_pa(q,p)\dd\zeta
\\
+\frac{\hbar^3}{\pi^d}\!\!\!\int_{\bR^d}\!\!\left(\hat\scrV(t,-2\zeta)e^{-2iq\cdot\zeta}\!-\!\hat\scrV(t,2\zeta)e^{2iq\cdot\zeta}\right)\left(\int_0^1\tfrac{(1-\tau)^2}{2}\grad^3_pa(q,p\!+\!\tau\hbar\zeta):\zeta^{\otimes 3}\dd\tau\right)\dd\zeta&\,.
\ea
\ee
By the Fourier inversion formula, the first integral on the right hand side of the equality above is
$$
\ba
-\frac\hbar{\pi^d}\grad_pa(q,p)\cdot\int_{\bR^d}\hat\scrV(t,2\zeta)e^{2iq\cdot\zeta}(2\zeta)\dd\zeta
\\
=i\hbar\grad_pa(q,p)\cdot\grad\scrV(t,q)=-i\hbar\{\scrV(t,\cdot),a\}(q,p)&\,.
\ea
$$
With \eqref{CommutDelta}  and \eqref{CommutV}, this proves the identity in the lemma.

With the expression of $R_{\hbar}$ so obtained, one has
$$
\ba
\d_q^\a\d_q^\b R_{\hbar}(t,q,p)
\\
=\tfrac{i}{\pi^d}\sum_{\a'\le\a}{\a\choose{\a'}}\int_{\bR^d}(2i\zeta)^{\a-\a'}\left(\hat\scrV(t,2\zeta)e^{2iq\cdot\zeta}-(-1)^{|\a-\a'|}\hat\scrV(t,-2\zeta)e^{-2iq\cdot\zeta}\right)\zeta^{\otimes 3}
\\
:\left(\int_0^1\tfrac{(1-\tau)^2}{2}\grad^3_p\d_q^{\a'}\d_p^\b a(q,p+\tau\hbar\zeta)\dd\tau\right)\dd\zeta&\,.
\ea
$$
Hence
$$
\ba
|\d_q^\a\d_q^\b R_{\hbar}(t,q,p)|\le\frac{1}{6\pi^d}\sup_{\a'\le\a}\|\grad^3_p\d_q^{\a'}\d_p^\b a\|_{L^\infty} \int_{\bR^d}2|\hat\scrV(t,2\zeta)|(1+2|\zeta|)^{|\a|}|\zeta|^3\dd\zeta
\\
\le\frac{2^{2+d}}{3\pi^d}\sup_{\a'\le\a}\|\grad^3_p\d_q^{\a'}\d_p^\b a\|_{L^\infty} \int_{\bR^d}|\hat\scrV(t,\xi)|(1+|\xi|)^{|\a|+3}\dd\xi&\,.
\ea
$$
By Calderon-Vaillancourt's theorem
$$
\|\OP^W_{\hbar}[R_{\hbar}(t,\cdot)]\|\le\tfrac{\g_d}{24}\sup_{|\a|\le[d/2]+1\atop |\b|\le[d/2]+4}\|\d_q^{\a}\d_p^\b a\|_{L^\infty} \int_{\bR^d}|\hat\scrV(t,\xi)|(1+|\xi|)^{[d/2]+4}\tfrac{\dd\xi}{(2\pi)^d}\,.
$$
This completes the proof of Lemma \ref{L-Qh}.

\subsection{Conclusion}

\begin{proof} [Proof of Lemma \ref{L-[E,B]}]
Putting together the estimates in Lemmas \ref{L-EstimB} and \ref{L-scrB-B} shows that
$$
\ba
\left\|\frac1\hbar[E_\om,\scrB_{\hbar}(t,s)]\right\|\le\left\|\frac1\hbar[E_\om,\scrB_{\hbar}(t,s)-B_{\hbar}(t,s)]\right\|+\left\|\frac1\hbar[E_\om,B_{\hbar}(t,s)]\right\|
\\
\le\frac2\hbar\|\scrB_{\hbar}(t,s)-B_{\hbar}(t,s)\|+\left\|\frac1\hbar[E_\om,B_{\hbar}(t,s)]\right\|
\\
\le\tfrac{\g_d}{12}\cF_{d+5}\frac{e^{\Ga_2(d+5)|t-s|}-1}{\Ga_2(d+5)}\hbar^2\sup_{|\a|\le[d/2]+1\atop |\b|\le[d/2]+4}\|\d_q^{\a}\d_p^\b b\|_{L^\infty} \int_{\bR^d}\sup_{\tau\in\bR}|\hat\scrV(\tau,\xi)|(1+|\xi|)^{[d/2]+4}\tfrac{\dd\xi}{(2\pi)^d}
\\
+\g_d\cF_{d+3}e^{(d+3)\Ga_2|t-s|}|\om|\|b\|_{d+3,\infty}
\\
\le\g_d\cF_{d+5}\|b\|_{d+5,\infty}\left(|\om|e^{\Ga_2(d+5)|t-s|}+\hbar^2\,\frac{e^{\Ga_2(d+5)|t-s|}-1}{12\Ga_2(d+5)}\bW\right)
\ea
$$
and this inequality implies the result in Lemma \ref{L-[E,B]} with
$$
C_d[\bW]:=\max_{|\nu|\le d+5}F_\nu[\bW]
$$
since all the functions $z\mapsto F_\nu[z]$ are nondecreasing and 
$$
\Ga_m\le\sup_{\tau\in\bR}\int_{\bR^d}|\scrV(\tau,\xi)||\xi|^m\tfrac{\dd\xi}{(2\pi)^d}\le\bW\,,\quad\hbox{ for all }m=0,\ldots,d+6\,.
$$ 

\end{proof}

\smallskip
Finally, we apply Lemma \ref{L-[E,B]} to prove the key semiclassical bound for the commutator to be estimated in Lemma \ref{L-KeyL}.

\begin{proof}[Proof of Lemma \ref{L-KeyL}]
Apply Lemma \ref{L-[E,B]} to the case where 
$$
\scrV(t,x)=V_{R(t)}(x)=\int_{\bR^d}V(x-y)r(t,y,y)\dd y\,.
$$
Set
$$
\tilde\Ga_m:=\|\grad^mV\|_{L^\infty(\bR^d)}\le\bV<\infty\quad\hbox{ for }m=0,\ldots,d+6\,.
$$

Since $R(t)\in\cD(\fH)$ for all $t\in\bR$, the function $\rho(t,):\,y\mapsto r(t,y,y)$ satisfies
$$
\rho(t,y)\ge 0\hbox{  for a.e. }y\in\bR^d\,,\quad\hbox{ and }\int_{\bR^d}\rho(t,y)\dd y=1\,,\quad \hbox{ for all }t\in\bR\,.
$$
Hence
$$
\Ga_m:=\sup_{t\in\bR}\|\grad^m_x\scrV(t,\cdot)\|_{L^\infty(\bR^d)}=\sup_{t\in\bR}\|\rho(t,\cdot)\star\grad^mV\|_{L^\infty(\bR^d)}\le\tilde\Ga_m
$$
for all $m=0,\ldots,d=6$. In particular $\Ga_2\le\tilde\Ga_2$.

On the other hand, the Fourier transform $\hat\rho(t,\cdot)$ of the function $\rho(t,\cdot)$ satisfies
$$
\|\hat\rho(t,\cdot)\|_{L^\infty}\le 1\,,\quad\hbox{ for all }t\in\bR\,,
$$
so that
$$
|\hat\scrV(t,\xi)|=|\hat V(\xi)\hat\rho(t,\xi)|\le|\hat V(\xi)|\,,\quad\hbox{ for a.e. }\xi\in\bR^d\hbox{ and all }t\in\bR\,.
$$
Hence
$$
\bW=\int_{\bR^d}\sup_{\tau\in\bR}|\hat\scrV(\tau,\xi)|(1+|\xi|)^{d+6}\dd\xi\le\int_{\bR^d}|\hat V(\xi)|(1+|\xi|)^{d+6}\dd\xi=\bV\,.
$$

Since $A^{in}\in\cW[T]$, there exists $\tau\in[0,T]$ such that $A^{in}=S(0,\tau)\OP^W_{\hbar}[a]S(\tau,0)$ with $a\in\cB^{d+5}$. Hence 
$$
A(s)=S(s,\tau)\OP^W_{\hbar}[a]S(\tau,s)
$$
and Lemma \ref{L-[E,B]} implies that
$$
\frac1\hbar\|[E_\om,A(s)]\|\le C_d[\bW]\left(|\om|e^{(d+5)\Ga_2|s-\tau|}+\hbar^2\,\frac{e^{(d+5)\Ga_2|s-\tau|}-1}{(d+5)\Ga_2}\bW\right)\,.
$$
Since $s,\tau\in[0,T]$, one has $|s-\tau|\le T$, so that
$$
\frac1\hbar\|[E_\om,A(s)]\|\le C_d[\bW]\left(|\om|e^{(d+5)\Ga_2T}+\hbar^2\,\frac{e^{(d+5)\Ga_2T}-1}{(d+5)\Ga_2}\bW\right)\,.
$$
Finally, this inequality implies that
$$
\frac1\hbar\|[E_\om,A(s)]\|\le C_d[\bV]\left(|\om|e^{(d+5)\tilde\Ga_2T}+\hbar^2\,\frac{e^{(d+5)\tilde\Ga_2T}-1}{(d+5)\tilde\Ga_2}\bV\right)\,.
$$
because the functions
$$
z\mapsto e^{cz}\quad\hbox{ and }\quad z\mapsto\frac{e^{cz}-1}{z}
$$
are increasing on $(0,+\infty)$ for all $c>0$, while $W\mapsto C_d[W]$ is increasing on $(0,+\infty)$. This last inequality is exactly the bound in Lemma \ref{L-KeyL}.
\end{proof}


\begin{appendix}


\section{A Unified Setting for the Terms\\ $\ad^*(-\tfrac12\hbar^2\Dlt)\cM_N(t)$ and $\cC[V,\cM_N(t),\cM_N(t)$}\label{A-Hochschild}


According to Lemma \ref{L-LRcP}, 
$$
\ad^*(-\tfrac12\hbar^2\Dlt)\scrR_N(t)-\cC[V,\scrR_N(t),\scrR_N(t)]=\ad^*(-\tfrac12\hbar^2\Dlt+V_{R(t)})\scrR_N(t)\,,
$$
when $\scrR_N(t)$ is of the form \eqref{HartreeEM}. 

This suggests to look for a common structure in both terms $\ad^*(-\tfrac12\hbar^2\Dlt)\cM_N(t)$ and $\cC[V,\cM_N(t),\cM_N(t)]$, although $\cM_N(t)$ is not in the form \eqref{HartreeEM}. This can be achieved, at least at the formal level,
by the following remark.

All the tensor products appearing in the discussion below designate tensor products of $\bC$- vector spaces. Observe that $\cL(\cL(\fH),\cL(\fH_N))$ is a $\cL(\fH)^{op}\otimes\cL(\fH_N)$-bimodule, where $\cL(\fH)^{op}$ is the opposite 
of the algebra $\cL(\fH)$ --- in other words, the product in $\cL(\fH)^{op}$ is 
$$
\mu^{op}:\,\cL(\fH)\otimes\cL(\fH)\ni A\otimes B\mapsto BA\in\cL(\fH)\,.
$$
The product in the algebra $\cL(\fH)^{op}\otimes\cL(\fH_N)$ is defined by the formula
$$
(P\otimes Q_N)(P'\otimes Q'_N):=(P'P)\otimes(Q_NQ'_N)\,,\qquad P,P'\in\cL(\fH)\,,\quad Q_N,Q'_N\in\cL(\fH_N)\,.
$$
The left scalar multiplication in the $\cL(\fH)^{op}\otimes\cL(\fH_N)$-bimodule $\cL(\cL(\fH),\cL(\fH_N))$ is defined as follows. Let $\L\in\cL(\cL(\fH),\cL(\fH_N))$, let $P\in\cL(\fH)$ and $Q_N\in\cL(\fH_N)$, then
$$
(P\otimes Q_N)\cdot\L:\,\cL(\fH)\ni A\mapsto Q_N\L(PA)\in\cL(\fH_N)\,.
$$
That this is a left action of $\cL(\fH)^{op}\otimes\cL(\fH_N)$ on $\cL(\cL(\fH),\cL(\fH_N))$ is seen with the following elementary computation
$$
\ba
((P'\otimes Q'_N)\cdot((P\otimes Q_N)\cdot\L))A&=Q'_N((P\otimes Q_N)\cdot\L)(P'A)
\\
&=(Q'_NQ_N)\L(PP'A)
\\
&=(((PP')\otimes(Q'_NQ_N))\L)A\,,
\ea
$$
for all $P,P'$ and $A\in\cL(\fH)$ and $Q_N,Q'_N\in\cL(\fH_N)$. Likewise, one easily checks that the right scalar multiplication in the $\cL(\fH)^{op}\otimes\cL(\fH_N)$-bimodule $\cL(\cL(\fH),\cL(\fH_N))$, defined by the formula
$$
\L\cdot(P\otimes Q_N):\,\cL(\fH)\ni A\mapsto \L(AP)Q_N\in\cL(\fH_N)
$$
is indeed a right action.

In this setting, one has
\be\lb{Hochsch-b}
\L\cdot(P\otimes Q_N)-(P\otimes Q_N)\cdot\L=\bbb_1(\L\otimes(P\otimes Q_N))\,,
\ee
where 
$$
\bbb_1:\,\cL(\cL(\fH),\cL(\fH_N))\otimes(\cL(\fH)^{op}\otimes\cL(\fH_N))\to\cL(\cL(\fH),\cL(\fH_N))
$$
is the Hochschild boundary map in degree one, defined on the space of Hochschild $1$-chains
$$
C_1(\cL(\cL(\fH),\cL(\fH_N));\cL(\fH)^{op}\otimes\cL(\fH_N)):=\cL(\cL(\fH),\cL(\fH_N))\otimes(\cL(\fH)^{op}\otimes\cL(\fH_N))
$$
with values in the space of Hochschild $0$-chains
$$
C_0(\cL(\cL(\fH),\cL(\fH_N));\cL(\fH)^{op}\otimes\cL(\fH_N)):=\cL(\cL(\fH),\cL(\fH_N))\,.
$$
(See for instance section 1.1 in \cite{Loday} for a quick presentation of Hochschild homology.) Hence
\be\lb{CC=b1}
\ba
\cC[E_\om,\cM_N,\cM_N]=&(\cM_NE^*_\om)\cM_N(E_\om A)-(\cM_N(AE_\om))(\cM_NE^*_\om)
\\
=&-\bbb_1(\cM_N\otimes(E_\om\otimes(\cM_NE^*_\om)))&\,.
\ea
\ee

In the special case where $Q_N=I_{\fH_N}$, the identity \eqref{Hochsch-b} takes the form
\be\lb{b1=ad}
\bbb_1(\L\otimes(P\otimes I_{\fH_N})=\ad^*(P)\L\,,
\ee
and this is the key observation leading to Lemma \ref{L-LRcP}.

Since the Hochschild boundary map $\bbb_1$ is linear, one can think of the operator $\cC[V,\cM_N,\cM_N]$ as 
\be\lb{CCV=b1}
\cC[V,\cM_N,\cM_N]=-\bbb_1\left(\cM_N\otimes\int_{\bR^d}E_\om\otimes(\cM_NE^*_\om)\hat V(\dd\om)\right)\,.
\ee
With the previous identity, equation \eqref{EMQEq} takes the form
\be\lb{EMQEqb1}
i\hbar\d_t\cM_N(t)=\bbb_1(\cM_N(t)\otimes H[\cM_N(t)])
\ee
where $\scrH[\cM_N(t)]$ is the $N$-body quantum Hamiltonian viewed as the element of $\cL(\fH)^{op}\otimes\cL(\fH_N)$ defined by the formula
\be\lb{HN[MN]}
\scrH[\cM_N(t)]:=-\tfrac12\hbar^2\Dlt\otimes I_{\fH_N}+\int_{\bR^d}E_\om\otimes(\cM_NE^*_\om)\hat V(\dd\om)\,.
\ee
While the discussion in the previous paragraphs is mathematically rigorous, writing \eqref{EMQEq} as \eqref{EMQEqb1} is purely formal for two reasons. First, $\Dlt$ is an unbounded operator on $\fH$, so that one cannot think of
$\cM_N(t)\otimes(-\tfrac12\hbar^2\Dlt\otimes I_{\fH_N})$ as an element of $C_1(\cL(\cL(\fH),\cL(\fH_N));\cL(\fH)^{op}\otimes\cL(\fH_N))$. The term
$$
\int_{\bR^d}E_\om\otimes(\cM_NE^*_\om)\hat V(\dd\om)
$$
is even more annoying. Unless $\hat V$ is a finite linear combination of Dirac measures, this integral does not define an element of  $\cL(\fH)^{op}\otimes\cL(\fH_N)$. At best, this integral could be thought of as an element of some 
completion of $\cL(\fH)^{op}\otimes\cL(\fH_N)$, the choice of which might not be completely obvious. (For instance, using cross norms to define the topology on $\cL(\fH)^{op}\otimes\cL(\fH_N)$ might lead to serious difficulties since
the map $\bR\ni\om\mapsto E_\om\in\cL(\fH)$ is not continuous for the norm topology. Likewise, one should probably avoid thinking of the integral above as a Bochner integral since $\cL(\fH)$ and $\cL(\fH_N)$ are not separable.)
In view of these difficulties, we shall not go further in this direction, and stick to our Definition \ref{D-DefC} of the interaction term $\cC[V,\cM_N,\cM_N]$. Nevertheless, it might be useful to keep in mind the form \eqref{EMQEqb1} of
the governing equation \eqref{EMQEq} for $\cM_N$, together with the formal expression of the interaction term as in \eqref{CCV=b1}-\eqref{HN[MN]}.

Of course, if $\cM_N(t)$ is of the form $\scrR(t)$ given by \eqref{HartreeEM}, then
$$
\int_{\bR^d}E_\om\otimes(\cM_NE^*_\om)\hat V(\dd\om)=\left(\int_{\bR^d}E_\om\Tr_\fH(R(t)E^*_\om)\hat V(\dd\om)\right)\otimes I_{\fH_N}=V_{R(t)}\otimes I_{\fH_N}\,,
$$
so that
$$
\scrH[\scrR(t)]=(-\tfrac12\hbar^2\Dlt+V_{R(t)})\otimes I_{\fH_N}\,.
$$
Then, according to \eqref{b1=ad}, 
$$
\bbb_1(\scrR(t)\otimes\scrH[\scrR(t)])=\ad^*(-\tfrac12\hbar^2\Dlt+V_{R(t)})\scrR(t)\,,
$$
so that equation \eqref{EMQEqb1} for $\scrR$ given by \eqref{HartreeEM} takes the form
$$
i\hbar\d_t\scrR(t)=\bbb_1(\scrR(t)\otimes\scrH[\scrR(t)])=\ad^*(-\tfrac12\hbar^2\Dlt+V_{R(t)})\scrR(t)\,,
$$
which is clearly equivalent to \eqref{Hartree} for $R(t)$.


\section{Wigner Transformation, Weyl Quantization\\ and the Calderon-Vaillancourt Theorem}\label{A-Weyl}


We first recall the notion of Wigner transform \cite{Wigner}. For each $K\in\cL(\fH)$ with integral kernel $k\equiv k(x,y)\in\bC$ such that $k\in\cS(\bR^d\times\bR^d)$, the Wigner transform at scale $\hbar$ of $K$ is the element of 
$\cS(\bR^d\times\bR^d)$ defined by the formula
$$
W_{\hbar}[K](x,\xi):=\tfrac1{(2\pi)^d}\int_{\bR^d}k(x+\tfrac12\hbar y,x-\tfrac12\hbar y)e^{-i\xi\cdot y}\dd y\,.
$$
By Plancherel's theorem, for each $K_1,K_2\in\cL(\fH)$ with integral kernels $k_1,k_2\in\cS(\bR^d\times\bR^d)$, one has
$$
\ba
\iint_{\bR^d\times\bR^d}\overline{W_{\hbar}[K_1](x,\xi)}W_{\hbar}[K_2](x,\xi)\dd x\dd\xi
\\
=\frac1{(2\pi\hbar)^d}\iint_{\bR^d\times\bR^d}\overline{k_1(x,y)}k_2(x,y)\dd x\dd y=\frac1{(2\pi\hbar)^d}\Tr_{\fH}(K^*_1K_2)&\,.
\ea
$$
Therefore the linear map $K\mapsto(2\pi\hbar)^{d/2}W_{\hbar}[K]$ extends as an unitary isomorphism from $\cL^2(\fH)$ to $L^2(\bR^d\times\bR^d)$. Observe that 
$$
W_{\hbar}[K^*](x,\xi)=\overline{W_{\hbar}[K](x,\xi)}\,,\quad\hbox{ for a.e. }(x,\xi)\in\bR^d\times\bR^d\,.
$$

We next recall the semiclassical Weyl quantization. For each $a\in\cS(\bR^d_x\times\bR^d_\xi)$ with polynomial growth as $|x|+|\xi|\to+\infty$, the expression
$$
(\OP^W_{\hbar}[a]\psi)(x)=\tfrac1{(2\pi)^d}\iint_{\bR^d\times\bR^d}a\left(\frac{x+y}2,\hbar\xi\right)e^{i\xi\cdot(x-y)}\psi(y)\dd\xi\dd y
$$
defines a linear map from $\cS(\bR^d)$ to itself. Observe that
$$
\int_{\bR^d}\overline{\phi(x)}(\OP^W_{\hbar}[a]\psi)(x)\dd x=\iint_{\bR^d\times\bR^d}\overline{W_{\hbar}[|\phi\ra\la\psi|](x,\xi)}a(x,\xi)\dd x\dd\xi\,.
$$
This formula extends to density operators more general than $|\psi\ra\la\psi|$, viz.
\be\lb{DualWOpW}
\Tr_{L^2(\bR^d)}(R\,\OP^W_{\hbar}[a])=\iint_{\bR^d\times\bR^d}W_{\hbar}[R](x,\xi)a(x,\xi)\dd x\dd\xi
\ee
for all $R\in\cD(\fH)$.

This defines $\OP^W_{\hbar}$ by duality as a linear map from $\cS(\bR^d)$ to $\cS'(\bR^d)$ for all $a\in\cS'(\bR^d\times\bR^d)$. Observe that
$$
\OP^W_{\hbar}[a]^*=\OP^W_{\hbar}[\overline a]\,.
$$

We shall need the formula for the composition of Weyl operators: one has
$$
\OP^W_{\hbar}[a]\OP^W_{\hbar}[b]=\OP^W_{\hbar}[a\#b]\,,
$$
where $a\#b$ is the Moyal product of $a$ and $b$, defined by the following formula:
$$
a\#b(q,p):=\frac1{(\pi\hbar)^{2d}}\int_{\bR^{4d}}a(q+x,p+\xi)b(q+y,p+\eta)e^{2i(x\cdot\eta-y\cdot\xi)/\hbar}\dd x\dd y\dd\xi\dd\eta\,.
$$

\begin{Thm}[Calderon-Vaillancourt]\lb{T-CV}
For each integer $d\ge 1$, there exists a constant $\g_d>0$ with the following property.

Let $a\in\cS'(\bR^d_x\times\bR^d_\xi)$ satisfy the condition
$$
|\a|,|\b|\le[d/2]+1\Rightarrow\d^\a_x\d^\b_\xi a\in L^\infty(\bR^d_x\times\bR^d_\xi)\,.
$$
Then, for all $\hbar\in(0,1)$, one has
$$
\|\OP^W_{\hbar}[a]\|_{\cL(L^2(\bR^d))}\le\g_d\max_{|\a|,|\b|\le[d/2]+1}\|\d^\a_x\d^\b_\xi a\|_{L^\infty(\bR^d\times\bR^d)}\,.
$$
\end{Thm}

This variant of the Calderon-Vaillancourt theorem is due to Boulkhemair \cite{cvbr} --- see also formula (54) on p. 236 in \cite{BR}.

\end{appendix}



\begin{thebibliography}{99}



\bibitem{APPP1}
A. Athanassoulis, T. Paul, F. Pezzotti, M. Pulvirenti: 
\textit{Strong Semiclassical Approximation of Wigner Functions for the Hartree Dynamics}, 
Rend. Lincei Mat. e Appl., \textbf{22} (2011), 525--552.
  
\bibitem{APPP2}  
A. Athanassoulis, T. Paul, F. Pezzotti, M. Pulvirenti: 
\textit{Semiclassical Propagation of Coherent States for the Hartree Equation},  
Ann. H. Poincar\'e, \textbf{12} (2011), 1613--1634.

\bibitem{BardosFGMauser}
C. Bardos, F. Golse, N. Mauser:
\textit{Weak coupling limit of the $N$-particle Schr\"odinger equation},
Methods Appl. Anal. \textbf{7} (2000), 275--293.

\bibitem{BEGMY}
C. Bardos, L. Erd\"os, F. Golse, N. Mauser, H.-T. Yau:
\textit{Derivation of the Schr\"odinger-Poisson equation from the quantum $N$-body problem},
C. R. Acad. Sci. Paris, S\'er. I  \textbf{334} (2002), 515--520.

\bibitem{BeneJakPortaSaffiSchlein} 
N Benedikter, V Jaksic, M Porta, C Saffirio, B Schlein: 
\textit{Mean-field evolution of fermionic systems}, 
J. Stat. Physics, \textbf{166} (2017), 1345--1364.

\bibitem{BPSS}
N. Benedikter, M. Porta, C. Saffirio, B. Schlein: 
\textit{From the Hartree dynamics to the Vlasov equation}, 
Arch. Rational Mech. Anal. \textbf{221} (2016), 273--334.

\bibitem{BenePortaSchlein}
N. Benedikter, M. Porta  and B. Schlein: 
``Effective Evolution Equations from Quantum Dynamics'', 
Springer Briefs in Mathematical Physics (2016).

\bibitem{cvbr}
A. Boulkhemair:
\textit{$L^2$ estimates for Weyl quantization},
J. Functional Anal. \textbf{165} (1999), 173--204.

\bibitem{BR} 
A. Bouzouina, D. Robert:
\textit{Uniform semiclassical estimates for the propagation of quantum observables},
Duke Math. J., \textbf{111} (2002), 223--252.

\bibitem{BraunHepp}
W. Braun, K. Hepp:
\textit{The Vlasov dynamics and its fluctuations in the 1/N limit of interacting classical particles},
Commun. Math. Phys. \textbf{56} (1977), 101--113.

\bibitem{ChenLeeSchlein}
L. Chen, J. Oon Lee, B. Schlein:
\textit{Rate of Convergence Towards Hartree Dynamics}, 
J. Stat. Phys. \textbf{144} (2011), 872--903.

\bibitem{Dobrushin}
R.L. Dobrushin: 
\textit{Vlasov equations},
Funct. Anal. Appl. \textbf{13} (1979), 115--123.

\bibitem{ErdosYau} 
L. Erd\"os, H.-T. Yau:
\textit{Derivation of the nonlinear Schr\"odinger equation from a many body Coulomb system}, 
Adv. Theor. Math. Phys. \textbf{5} (2001), 1169--1205.

\bibitem{FrohGraffiSchwartz}
J. Fr\"olich, S. Graffi and S. Schwartz: 
\textit{Mean-Field- and Classical Limit of Many-Body Schr\"odinger Dynamics for Bosons},
Commun. Math. Phys. \textbf{271} (2007), 681--697.

\bibitem{GV} 
J. Ginibre, G. Velo: 
\textit{The classical field limit of scattering theory for nonrelativistic many-boson systems. I-II}, 
Commun. Math. Phys. \textbf{66} (1979), 37-76 and \textbf{68} (1979), 451-768.

\bibitem{FGMouPaul}
F. Golse, C. Mouhot, T. Paul:
\textit{On the Mean Field and Classical Limits of Quantum Mechanics},
Commun. Math. Phys. \textbf{343} (2016), 165--205.

\bibitem{FGMouhotRicci}
F. Golse, C. Mouhot, V. Ricci:
\textit{Empirical measures and mean-field hierarchies},
Kinet. Relat. Models \textbf{6} (2013), 919--943.

\bibitem{FGPaul}
F. Golse, T. Paul:
\textit{The Schr\"odinger Equation in the Mean-Field and Semiclassical Regime},
Arch. Rational Mech. Anal. \textbf{223} (2017), 57--94.

\bibitem{FGPaulCohSta}
F. Golse, T. Paul:
\textit{Wave Packets and the Quadratic Monge-Kantorovich Distance in Quantum Mechanics},
arXiv:1707.04161v1 [math.AP].

\bibitem{FGPaulPulvi}
F. Golse, T. Paul, M. Pulvirenti:
\textit{On the derivation of the Hartree equation in the mean-field limit: uniformity in the Planck constant},
arXiv:1606.06436 [math.AP].

\bibitem{GraffiMartiPulvi}
S. Graffi, A. Martinez, M. Pulvirenti:
\textit{Mean-field approximation of quantum systems and classical limit},
Math. Models Methods Appl. Sci. \textbf{13} (2003), 59--73.

\bibitem{Hartree1928}
D.R. Hartree:
\textit{The Wave Mechanics of an Atom with a Non-Coulomb Central Field. Part II. Some Results and Discussions},
Proc. Cambridge Phil. Soc. \textbf{24} (1928), 111-132.

\bibitem{HaurayJabin1}
M. Hauray, P.-E. Jabin:
\textit{$N$-particle Approximation of the Vlasov Equations with Singular Potential},
Arch. Rational Mech. Anal. \textbf{183} (2007), 489—524.

\bibitem{HaurayJabin2}
M. Hauray, P.-E. Jabin:
\textit{Particle Approximation of Vlasov Equations with Singular Forces},
Ann. Scient. Ecole Norm. Sup. \textbf{48} (2015), 891--940.

\bibitem{Hepp74}
K. Hepp:
\textit{The classical limit for quantum mechanical correlation functions}, 
Commun. Math. Phys. \textbf{35} (1974), 

\bibitem{KnowlesPickl} 
A. Knowles, P. Pickl: 
\textit{Mean-Field Dynamics: Singular Potentials and Rate of Convergence}, 
Commun. Math. Physics \textbf{298} (2010), 101--138.

\bibitem{Lazaro}
D. Lazarovici:
\textit{The Vlasov-Poisson dynamics as the mean-field limit of rigid charges},
Comm. Math. Phys. \textbf{347} (2016), 271--289.

\bibitem{PicklLazaro}
D. Lazarovici, P. Pickl:
\textit{A mean-field limit for the Vlasov-Poisson system},
 Arch. Ration. Mech. Anal. \textbf{225} (2017), 1201--1231.

\bibitem{LionsPaul}
P.-L. Lions, T. Paul: 
\textit{Sur les mesures de Wigner},
Rev. Math. Iberoam. \textbf{9} (1993), 553--618.

\bibitem{Loday}
J.-L. Loday:
``Cyclic homology'',
Springer-Verlag,  Berlin, Heidelberg, 1992.

\bibitem{MischMouWenn}
S. Mischler, C. Mouhot, B. Wennberg:
\textit{A new approach to quantitative propagation of chaos for drift, diffusion and jump processes},
Probab. Theory Related Fields \textbf{161} (2015), 1--59.

\bibitem{NarnhoSewell}
H. Narnhofer, G. Sewell:
\textit{Vlasov hydrodynamics of a quantum mechanical model}, 
Commun. Math. Phys. \textbf{79} (1981), 9--24.

\bibitem{NeunzertWick}
H. Neunzert, J. Wick:
\textit{Die Approximation der L\"osung von Integro-Differentialgleichungen durch endliche Punktmengen},
Lecture Notes in Mathematics, vol. 395, pp. 275--290. 
Springer, Berlin (1974).

\bibitem{PezzoPulvi}
F. Pezzoti, M Pulvirenti: 
\textit{Mean-Field Limit and Semiclassical Expansion of a Quantum Particle System},  
Ann. Henri Poincar\'e \textbf{10} (2009), 145--187.

\bibitem{Pickl}
P. Pickl:
\textit{A simple derivation of mean-field limits for quantum systems}, 
Lett. Math. Phys. \textbf{97} (2011), 151--164.

\bibitem{RodSchlein}
I. Rodnianski, B. Schlein:
\textit{Quantum fluctuations and rate of convergence towards mean-field dynamics}, 
Commun. Math. Phys. \textbf{291} (2009), 31--61.

\bibitem{Spohn80}
H. Spohn:
\textit{Kinetic equations from Hamiltonian dynamics},
Rev. Mod. Phys. \textbf{52} (1980), 600--640.

\bibitem{Wigner}
E.P. Wigner:
\textit{On the quantum correction for thermodynamic equilibrium}, 
Phys. Rev. \textbf{40} (1932) 749--759.

\end{thebibliography}
\end{document}